\documentclass[12pt,a4paper,reqno]{amsart}
\usepackage{amsmath,amscd,amssymb,latexsym}
\usepackage{longtable}

\newtheorem{Thm}{Theorem}[section]
\newtheorem{Cor}[Thm]{Corollary}
\newtheorem{Lem}[Thm]{Lemma}
\newtheorem{Prop}[Thm]{Proposition}
\theoremstyle{definition}
\newtheorem{Def}[Thm]{Definition}
\newtheorem{Exm}[Thm]{Example}

\newtheorem{Rk}[Thm]{Remark}

\def \ZZ{{\mathbb{Z}}}
\def \QQ{{\mathbb{Q}}}
\def \RR{{\mathbb{R}}}
\def \CC{{\mathbb{C}}}
\def \FF{{\mathbb{F}}}
\def \PP{{\mathbb{P}}}

\def \TTT{{\mathcal{T}}}

\def \qq{{\mathfrak q}}

\def\det{\mathop{\mathrm{det}}}

\def\Ell{\mathop{\mathrm{Ell}}}
\def\mod{\mathop{\mathrm{mod}}}

\def\vert{\mathop{\mathrm{vert}}}
\def\edge{\mathop{\mathrm{edge}}}
\def\star{\mathop{\mbox{\Large$*$}}}

\def\Cl{\mathop{\mathrm{Cl}}}
\def\Cusp{\mathop{\mathrm{Cusp}}}

\def\Gamma{\varGamma}
\def\Delta{\varDelta}

\begin{document}

\begin{center}
\bf{\LARGE{Quasi-inner automorphisms of Drinfeld modular groups}}
\end{center}

\title[Quasi-inner automorphisms]{{}}
\author[A.~W.~Mason and A. Schweizer]{A.~W.~Mason and Andreas Schweizer}

\address{A.~W.~Mason, 
Department of Mathematics, University of Glasgow,
Glasgow G12 8QW, Scotland, U.K.}
\email{awm@maths.gla.ac.uk}

\address{Andreas Schweizer,
Am Felsenkeller 61,
78713 Schramberg, 
Germany}

\begin{abstract}
 Let $A$ be the set of elements in an algebraic function field $K$ over $\FF_q$ which are integral outside a fixed place $\infty$. Let $G=GL_2(A)$ be a {\it Drinfeld modular group}. The normalizer of $G$ in $GL_2(K)$,
 where $K$ is the quotient field of $A$, gives rise to automorphisms of $G$, which we refer to as {\it quasi-inner}. Modulo the inner automorphisms of $G$ they form a group $Quinn(G)$ which is isomorphic to $\Cl(A)_2$, the $2$-torsion in the ideal class group $\Cl(A)$. \\
 \noindent The group $Quinn(G)$ acts on all kinds of objects associated with $G$. For example, it acts freely on the cusps and elliptic points of $G$. If $\TTT$ is the associated Bruhat-Tits tree the elements of $Quinn(G)$ induce non-trivial automorphisms of the quotient graph $G \backslash \TTT$, generalizing an earlier result of Serre. It is known that the ends of $G \backslash \TTT$ are in one-one correspondence with the cusps of $G$. Consequently $Quinn(G)$ acts freely on the ends. In addition $Quinn(G)$ acts transitively on the those ends which are in one-one correspondence with the vertices of $G \backslash \TTT$ whose stabilizers are isomorphic to $GL_2(\FF_q)$. 
\end{abstract}

\maketitle 

\section{Introduction}

Let $K$ be an algebraic function field of one variable with constant
field $\FF_q$, the finite field of order $q$. Let $\infty$ be a
fixed place of $K$ and let $\delta$ be its degree. The ring $A$ of
all those elements of $K$ which are integral outside $\infty$ is a Dedekind domain. Denote by $K_{\infty}$ the completion of $K$ with respect to 
$\infty$ and let $C_{\infty}$ be the $\infty$-completion of an algebraic closure of $K_{\infty}$. The group $GL_2(K_{\infty})$ (and its subgroup $G=GL_2(A)$) act as M\"obius transformations on $C_{\infty}$, 
$K_{\infty}$ and hence $\Omega = C_{\infty} \backslash K_{\infty}$,
the \it Drinfeld upper halfplane. \rm   This is part of a far-reaching analogy, initiated by Drinfeld \cite{Drinfeld}, where $\QQ, \RR, \CC$ are replaced by $K,K_{\infty}, C_{\infty}$, respectively. The roles of the classical upper half plane (in $\CC$) and the classical modular group $SL_2(\ZZ)$ are assumed by $\Omega$ and $G$, repspectively. \\
\noindent Modular curves, that is quotients of the complex upper half plane by finite index subgroups of $SL_2(\ZZ)$, are an indispensable tool when proving deep theorems about elliptic curves. Of similar importance in the theory of Drinfeld $A$-modules of rank $2$ are \it Drinfeld modular curves, \rm which are (the ``compactifications" of) the quotient spaces $H \backslash \Omega$, where $H$ is a finite index subgroup of $G$. Consequently we refer to $G$ as a \it Drinfeld modular group. \rm \\
\noindent A complicating factor in this correspondence between $SL_2(\ZZ)$ and $G$ is that, while the genus of the former is zero, for different choices of $K$ and $\infty$ the genus of $G$ can take many values. The simplest case, where $K=\FF_q(t)$ and $A=\FF_q[t]$ (equivalently $g=0$ and $\delta=1$), has to date attracted most attention. \\
\noindent An element $\omega \in \Omega$  which is stabilized by a non-scalar matrix in $G$ is called {\it elliptic}. Let $E(G)$ be the set of all such elements. It is known \cite[p.50]{Gekeler} that $E(G) \neq \emptyset$ if and only if $\delta$ is odd. Clearly $G$ acts on $E(G)$ and the elements of the set of $G$-orbits, $\Ell(G)= G\backslash E(G)=\{G\omega:\omega \in E(G)\}$, are called the {\it elliptic points} of $G$. It is known \cite[p.50]{Gekeler} that $\Ell(G)$ is finite. \\
 \noindent In addition $G$ acts on $\PP^1(K)=K \cup \{\infty\}$. (Here, of course, $\infty$ refers to the one point
compactification of $K$.) We refer to the elements of $\PP^1(K)$ as {\it rational points}. For each finite index subgroup, $H$, of $G$ the elements of $\Cusp(H)=H\backslash \PP^1(K)$ are called the {\it cusps} of $H$. Since $A$ is a Dedekind domain it is well-known that $\Cusp(G)$ can be identified with $\Cl(A)$, the {\it ideal class group} of $A$. As M\"obius transformations $G$ acts without inversion on $\TTT$, the  Bruhat-Tits tree associated with $GL_2(K_{\infty})$ and the \it ends \rm of the quotient graph $G \backslash \TTT$ are determined by $\Cusp(G)$ \cite[Theorem 9, p.106]{Serre}.\\ 
\noindent Cusps and elliptic points are important for several reasons. If $H$ is a finite index subgroup of $G$, the quotient space $H \backslash \Omega$ will, after adding $\Cusp(H)$, be the $C_{\infty}$-analog of a compact Riemann surface, which is called the \it Drinfeld modular curve \rm associated with $H$. Moreover, in the covering of Drinfeld modular curves induced by the natural map $H \backslash \Omega \rightarrow G \backslash \Omega$ ramification can only occur above the cusps and elliptic points of $G$. Also, for (classical and Drinfeld) modular forms, analyticity at the cusps and elliptic points requires special care.\\

 \noindent This paper is a continuation and extension of \cite{MSelliptic} which is concerned with the elliptic points of $G$. There the starting point 
 \cite[p.51]{Gekeler} is the existence of a bijection between $\Ell(G)$
 and $\ker \overline{N}$, where $\overline{N}: \Cl(\widetilde{A}) \rightarrow \Cl(A)$ is the norm map and $\widetilde{A}=A.\FF_{q^2}$. It can be shown \cite{MSelliptic} that $\Cl(\widetilde{A})_2 \cap \ker \overline{N}$, the $2$-torsion subgroup of $\ker \overline{N}$, is in bijection with $\Ell(G)^= =\{G\omega: \omega \in E(G),\;G\omega=G\overline{\omega}\}$,  where $\overline{\omega}$, the \it conjugate \rm of $\omega$, is the image of $\omega$ under the Galois automorphism of $K.\FF_{q^2}/K$. (In \cite{MSelliptic} $\Ell(G)^=$ is denoted by $\Ell(G)_2$.) Here we show that, when $\delta$ is odd, $\Cl(A)_2$ and the $2$-torsion in $\ker \overline{N}$ are isomorphic. This is the starting point for this paper where the principal focus of attention is the group $\Cl(A)_2$ and its actions on various objects related to $G$. \it Unless otherwise stated results hold for all $\delta$. \rm \\
 
\noindent Let $g \in N_{GL_2(K)}(G)$, the normalizer of $G$ in $GL_2(K)$.
Then $g$, acting by conjugation, induces an automorphism $\iota_g$ of $G$ which we refer to as \it quasi-inner. \rm If $g \in G.Z(K)$ then  $\iota_g$ reduces to an inner automorphism. If $g \in N_{GL_2(K)}(G) \backslash G.Z(K)$ we call $\iota_g$ \it non-trivial. \rm  We denote the quotient group 
$N_{GL_2(K)}(G)/G.Z(K)$ by $Quinn(G)$. It is well-known \cite{Cremona} that $Quinn(G)$ is isomorphic to $\Cl(A)_2$. Hence $G$ has non-trivial quasi-inner automorphisms if and only if $|\Cl(A)|$ is \it even. \rm Now, as an element of $GL_2(K)$, $\iota_g$ acts as a M\"obius transformation on the rational points and elliptic elements of $G$, as well as $\TTT$. In particular $\overline{g(\omega)}=g(\overline{\omega})$. Since all of these actions are trivial for scalar matrices they extend to actions of $Quinn(G)$ on $\Cusp(G)$, $\Ell(G)$ and the quotient graph, $G \backslash \TTT$. In this paper we study of the (often surprising) properties of these actions.
\newline
\begin{Thm}\label{freeint}  $Quinn(G)$ acts freely on \\
(i) $\Cusp(G)$, \\
(ii) $\Ell(G)$ ($\delta$ odd).
\end {Thm}
\noindent From the above it is clear that $Quinn(G)$ can be embedded as a subgroup $\Ell(G)^=$ (resp. $\Cl(A)_2$) of $\Ell(G)$ (resp. $\Cusp(G)$). We show that the action of $Quinn(G)$ is equivalent to multiplication by the elements of the subgroup. The ``freeness" in this result follows immediately.
Restricting to these subsets yields stronger results.

\begin{Cor}\label{freetransint} $Quinn(G)$ acts freely and transitively on \\ 
(i) $\Cl(A)_2$, \\ 
(ii) $\Ell(G)^=$ ($\delta$ odd). 

\end{Cor}
\begin{Cor}\label{4int}  When $\delta$ is odd $Quinn(G)$ acts freely on $\Ell(G)^{\neq}=\{G\omega:G\omega \neq G\overline{\omega}\}$. Moreover, if $\ker \overline{N}$ has no element of order $4$, then $Quinn(G)$ acts freely on $$\{\{G\omega,G\overline{\omega}\}: G\omega \in \Ell(G)^{\neq}\}.$$
\end{Cor}
\begin{Thm}\label{auto2int} Every non-trivial element of $Quinn(G)$ determines an automorphism of $G \backslash \TTT$ of order $2$ which preserves the structure of all its vertex and edge stabilizers.
\end{Thm}
\noindent Serre \cite[Exercise 2 e), p. 117]{Serre} states this result for the special case $K=\FF_q(t)$ with $\delta$ even. Our result shows that in general the quotient graph has symmetries of this type provided $|\Cl(A)|$ is even. (In general this restriction is necessary.)

\noindent We now list more detailed results on the action of $Quinn(G)$ on  $G \backslash \TTT$. Serre \cite[Theorem 9, p.106]{Serre} has described the basic structure of $G \backslash \TTT$. Its \it ends \rm (i.e. the equivalence classes of semi-infinite paths without backtracking) are in one-one correspondence with the elements of $\Cl(A)$. To date the only cases for which the precise structures of $G \backslash \TTT$ are known are $g=0$,
\cite{MasonJGT}, \cite{KMS}, and $g=\delta=1$, \cite{Takahashi}.
\begin{Thm}\label{endsint}   $Quinn(G)$ acts freely on the ends of $G \backslash \TTT$ and, in addition, transitively on
 the ends of $G \backslash \TTT$ corresponding to the elements of $\Cl(A)_2$,
\end{Thm}
\noindent  We show that the ends corresponding to $\Cl(A)_2$ are in one-one correspondence with those vertices whose stabilizers are isomorphic to $GL_2(\FF_q)$. (Each such vertex is ``attached" to the corresponding end.)  It is known \cite[Corollary 2.12]{MSstabilizer} that if $G_v$ contains a cyclic subgroup of order $q^2-1$ then $G_v \cong \FF_{q^2}^*\; \mathrm{or}\; GL_2(\FF_q)$. \\
\noindent The \it building map \rm \cite[p.41]{Gekeler} extends to a map $\lambda: \Ell(G) \rightarrow \vert(G \backslash \TTT)$. This map leads to another action of $Quinn(G)$ on the quotient graph.

\begin{Thm} (a) $Quinn(G)$ acts freely and transitively on 
$$\{\tilde{v} \in \vert(G\backslash \TTT): G_v \cong GL_2(\FF_q)\}.$$
(b) Suppose that $\delta$ is odd and that $\ker\overline{N}$ has no element of order $4$. Then $Quinn(G)$ acts freely on
$$\{\tilde{v} \in \vert(G\backslash \TTT): G_v \cong \FF_{q^2}^*\}.$$
\end{Thm}
\noindent As an illustration of our results, especially the existence of reflective symmetries as in Theorem \ref{auto2int}, we conclude with diagrams of two examples of $G \backslash \TTT$ for each of which $g=\delta=1$, the so called ``elliptic" case. For these we make use of Takahashi's paper \cite{Takahashi}. Special features of these cases include the following. For part (i) see \cite[Theorem 5.1]{MSelliptic}.
\begin{Cor}\label{isolint} Suppose that $\delta=1$. \\ \\
(i) The isolated (i.e. (graph) valency $1$) vertices of $G\backslash \TTT$ are precisely those whose stabilizers are isomorphic to $GL_2(\FF_q)$ or $\FF_{q^2}^*$. \\ \\
(ii) If $\ker \overline{N}$ has no element of order $4$ then $Quinn(G)$ acts freely on the isolated vertices of $G \backslash \TTT$.
\end{Cor}
\noindent By looking at the stabilizers in $G$ of the objects discussed above we obtain several statements about the action of $Quinn(G)$ on the conjugacy classes of certain types of subgroups of $G$. (See Sections $3$ and $5$.) \\
\noindent For convenience we begin with a list of notations which will be used
throughout this paper.
\\ \par
\begin{tabular}{ll}
$\FF_q$              & the finite field with $q=p^n$ elements;\\
$K$                      & an algebraic function field of one variable with constant field $\FF_q$;\\
$g$                 & the genus of $K$;\\
$\infty$               & a chosen place of $K$;\\
$\delta$               & the degree of the place $\infty$;\\
$A$                      & the ring of all elements of $K$ that are integral outside $\infty$;\\
$K_{\infty}$        & the completion of $K$ with respect to $\infty$;\\
$\Omega$           & Drinfeld's half-plane;\\
$\TTT$                  & the Bruhat-Tits tree of $GL_2(K_{\infty})$;\\
$G$                      & the Drinfeld modular group $GL_2(A)$;\\
$Gx$ & the orbit of $x$ under the action of $G$ on the object $x$;\\ 
$\widehat{G}$    & $GL_2(K)$;\\
$Z(K)$                      &  the set of scalar matrices in $\widehat{G}$;\\
$Z$        &  $Z(K)\cap G$;\\
$\widetilde{K}$   & the quadratic constant field extension $K.\FF_{q^2}$;\\
$\widetilde{A}$   & $A.\FF_{q^2}$, the integral closure of $A$ in $\widetilde{K}$;\\
$\Cl(R)$               & the ideal class group of the Dedekind ring $R$;\\
${\Cl}^0(F)$       & the divisor class group of degree $0$ of the function field $F$;\\
$\Cusp(G)$          & $G \backslash \PP^1(K)$, the set of cusps of $G$;\\
$E(A)$               & the set of elliptic elements of $G$: \\
$\Ell(G)$                & $G\backslash E(A)$, the set of elliptic points of $G$;\\
$\overline{\omega}$   & the image of $\omega \in E(A)$ under the Galois automorphism of $\widetilde{K}/K$; \\
$\Ell(G)^{=}$            & $\{G\omega: \omega \in E(A),G\omega=G\overline{\omega} \}$;  \\
$\Ell(G)^{\neq}$ & $\Ell(G) \backslash \Ell(G)^{=}$;\\
$S(s)$ & the stabilizer in a finite index subgroup $S$ (of $G$) of $s \in\PP^1(K)$;\\
$G^{\omega}$&  the stabilizer in $G$ of $\omega \in C_{\infty} \backslash K$; \\ 
$S_w$ & the stabilizer in $S$ of $w \in \vert(\TTT) \cup \edge(\TTT)$; \\
$\mathcal {H}$ & $\{H \leq G: H \cong GL_2(\FF_q) \}$ ; \\
$\mathcal{C}$ &  $\{C \leq G: C \cong \FF_{q^2}^*\}$ ; \\
$\mathcal{C}_{mf}$ & $\{ C \in \mathcal{C}: C \;\mathrm{maximally \; finite \; in}\; $G$\}$; \\
$\mathcal{C}_{nm}$ & $\mathcal{C} \backslash \mathcal{C}_{mf}$; \\
$\mathcal{V}$& $\{ \widetilde{v} \in \vert(G \backslash \TTT): G_v \in \mathcal{C}\}$; \\
\end{tabular}
 \space
\\

\section{Quasi-inner automorphisms}\label{sect:quasi-inner}

\noindent \rm  Let $F$ be any field containing $A$ (and hence $K$) and let $Z(F)$ denote the set of scalar matrices in $GL_2(F)$. 
We are interested here in automorphisms of $G$ arising from conjugation by a non-scalar element of $GL_2(F)$. We first show this 
problem reduces to $N_{\widehat{G}}(G)$, the normalizer of $G$ in $\widehat{G}=GL_2 (K)$. For each $x \in F$ we use $(x)$ as 
a shorthand for the fractional ideal $Ax$.
 
\begin{Lem}\label{normalizer}
Let $M_0 \in GL_2(F)$ normalize $G$. Then
$$ M_0 \in Z(F).N_{\widehat{G}}(G).$$
\end{Lem}

\begin{proof} 
Let
$$ M_0= \left[\begin{array}{cc}\alpha&\beta\\\gamma&\delta
\end{array}\right].$$
\noindent Suppose that $\gamma \neq 0$. Replacing $M_0$ with $\gamma^{-1}M_0$ we may assume that $\gamma=1$. Now
$$NT(1)N^{-1} \in G,$$
\noindent where $N=M_0^{\pm 1}$. It follows that $\det(M_0)$, $\alpha$, $\delta \in K$ and hence that 
$\beta= \alpha\delta -\det(M_0) \in K$. The proof for the case where $\gamma=0$ is similar. 
\end{proof} 

\noindent We state a special case $(n=2)$ of a result of Cremona \cite{Cremona} . 
 
\begin{Thm}\label{Misquasi}
Let
$$ M= \left[\begin{array}{cc}a&b\\c&d
\end{array}\right] \in \widehat{G}$$
\noindent and define
$$\qq(M):=(a)+(b)+(c)+(d).$$
Then $M \in N_{\widehat{G}}(G)$ if and only if
$$\qq(M)^2=(\varDelta),$$
\noindent where $\varDelta=\det(M)$. 
\end{Thm}
 
\begin{Cor}\label{M^2}
Let $M \in N_{\widehat{G}}(G)$ with $\varDelta=\det(M)$.
\begin{itemize}
\item[(i)] $ \varDelta^{-1}M^2 \in SL_2(A).$
\item[(ii)] If $\varDelta \in A^*$, then $M \in G$.
\end{itemize}
\end{Cor}

\begin{proof}
(i) By Theorem \ref{Misquasi} every entry of $M^2$ is in $\qq(M)^2 =(\varDelta)$.
\par 
For part (ii) let $x$ be any entry of $M$. Then $x^2 \in A$ by Theorem \ref{Misquasi} 
and so $ x \in A$, since $A$ is integrally closed. 
\end{proof} 

\noindent Another important consequence \cite{Cremona} of  Theorem \ref{Misquasi} is the following. 
 
\begin{Thm}\label{isoto2tors}
The map $M\mapsto\qq(M)$ induces an isomorphism
$$N_{\widehat{G}}(G)/Z(K).G \; \cong \;\mathrm{Cl}(A)_2,$$
\noindent where $\mathrm{Cl}(A)_2$ is the subgroup of all involutions in $\mathrm{Cl}(A)$. 
\end{Thm}
 
\begin{proof}
This is another special case ($n=2$) of a result in \cite{Cremona}. 
If $\left[\begin{array}{cc}a&b\\c&d
\end{array}\right] \in N_{\widehat{G}}(G)$, 
it can be shown \cite[Remarks $2$]{Cremona} that
$$(a)+(b)=(a)+(c)=(d)+(b)=(d)+(c)=\qq(M).$$
Consequently there is a map from $N_{\widehat{G}}(G)$ to ${\Cl}(A)_2$,
which turns out to be an isomorphism. 
\end{proof}

\begin{Def}\label{def:quasi-inner}
An automorphism $\iota_g$ of $G$ is called \bf quasi-inner \rm if
$$\iota_g(x)=gxg^{-1}\;(x \in G),$$
\noindent for some $ g \in N_{\widehat{G}}(G)$.
We call $\iota_g$ {\it non-trivial}
if $g \notin Z(K).G$, i.e. if $\iota_g$ does not act like an inner automorphism. We note that
$$\iota_{g_1}=\iota_{g_2} \Leftrightarrow  g_1g_2^{-1} \in Z(K).$$ 
\noindent Finally we define 
$$Quinn(G):= N_{\widehat{G}}(G)/Z(K).G\cong\Cl(A)_2.$$
\end{Def}
\noindent So $Quinn(G)$ is the group of quasi-inner automorphisms modulo the inner ones. We note that in particular all quasi-inner automorphisms of $G$ act like inner automorphisms if 
$|\Cl(A)|$ is odd.

\noindent Let $\mathrm{Cl}^0(K)$ be the {\it group of divisor classes of degree zero} \cite[p.186]{Stich}. 
It is known \cite[p.104]{Serre} that the following exact sequence holds

\begin{equation}\label{exactsequence}
0\; \rightarrow \; {\Cl}^0(K) \; \rightarrow\; \Cl(A) \; \rightarrow \; \ZZ/ \delta \ZZ\; \rightarrow\; 0.
\end{equation}

\medskip
\noindent Our next result is an immediate consequence of Theorem \ref{isoto2tors}. 

\begin{Cor}\label{exquasi-inner}
$G$ has non-trivial quasi-inner automorphisms if and only if
$$ |\mathrm{Cl}(A)|=\delta|\mathrm{Cl}^0(K)|\; is\;even.$$ 
\end{Cor}

\begin{Exm} We illustrate the results of this section with the simplest case $K= \FF_q(t)$, the rational function field over $\FF_q$. Then there exists a (monic) polynomial $\pi(t) \in \FF_q[t]$,
of degree $\delta$, irreducible over $\FF_q$, such that 
$$A= \left\{\frac{f}{\pi^m}: f \in k[t], \; m \geq 0, \; \deg f \leq \delta m\right\}.$$
\noindent It is known \cite[Theorem 5.1.15, p.193]{Stich} that here $\mathrm{Cl}^0(K)$ is trivial so that
$$\Cl(A) \cong \ZZ/\delta\ZZ.$$
\noindent Hence $G$ has non-trivial quasi-inner automorphisms if and only if $\delta$ is even. 
Hence here either $Quinn(G)$ is trivial or cyclic of order $2$. 

\noindent For a specific illustration of Theorem \ref{isoto2tors} we restrict further to $\delta=2$. In this case $\pi(t)=t^2+\sigma t+\tau$, where $\sigma \in \FF_q$ and $\tau \in \FF_q^*$. We begin with 
the $A$-ideal generated by $\pi^{-1}$ and $t\pi^{-1}$ which is {\it not} principal. 
Let $\pi(t)=tt'+\tau$ and put
$$g_0=\left[\begin{array}{lll} \tau& t\\[10pt]
-t' & 1\end{array}\right].$$ 

\noindent Then by Theorem \ref{Misquasi} $g_0 \in N_{\widehat{G}}(G)$ 
and from Theorem \ref{isoto2tors} we see $g_0 \notin Z(K).G$. Hence $g_0$ provides a generator of $\Cl(A)_2$.
\end{Exm}

\begin{Rk}\label{quasi-innersect}  \rm
From the theory of Jacobian varieties we know that the $2$-torsion in $\mathrm{Cl}^0(K)$ is bounded 
by $2^{2g}$, and even by $2^g$ if the characteristic of $K$ is $2$ \cite[Theorem 11.12]{Rosen}. 
Hence by the exact sequence (1) it follows that $|Quinn(G)|=|\mathrm{Cl}(A)_2 |\leq 2^{2g+1}$ (and 
 $\leq 2^{g+1}$, when $\mathrm{char}(K)=2$).
\par 
\noindent In odd characteristic we can easily find examples with $|\Cl(A)_2|=2^{2g}$, provided we are willing to accept a big constant field. Given a function field $F$ of genus $g$ with constant field $\FF_{p^r}$, just pick $q=p^{rn}$ such that all $2$-torsion points of $Jac(F)$ are $\FF_q$-rational and consider $K=F.\FF_q $. Then $\Cl^0(K)_2  \cong (\ZZ/ 2\ZZ)^{2g}$. Choosing a place $\infty$ of $K$ of odd degree $\delta$, from the exact sequence (1), we see that $|\mathrm{Cl}(A)_2 |=2^{2g}$. 
\par
\noindent Similarly in characteristic $2$ examples for which $|\Cl(A)_2|=2^g$ can be found by choosing $F$ suitably, namely $F$ has to be \it ordinary. \rm \\
\noindent Whether for even $\delta$ one can reach the bound $2^{2g+1}$ (resp. $2^{g+1}$) depends on whether or not the induced short exact sequence for the Sylow $2$-subgroup of $\mathrm{Cl}(A)$ splits or not.
\end{Rk}
\begin{Def}  Let $R,\;S$ be subgroups of a group $T$. We write
$$R \sim S$$
\noindent if and only if $R=S^t=tSt^{-1}$, for some $t\in T$. 
We put
$$R^T=\left\{R^t: t \in T \right\}.$$
Let $\mathcal{S}$ be a set of subgroups of $T$. We put
$$\mathcal{S}^G=\left\{ S ^G: S \in \mathcal{S} \right\}.$$
 \end{Def} 
 \noindent This paper is principally concerned with various actions of $Quinn(G)$. It is appropriate at this point to describe in detail the most important of these. Let $\iota_g$ be as above.\\ \\
\noindent (i) It is clear that $GL_2(K_{\infty})$ acts on $\Omega$ as 
M\"obius transformations and that this action is trivial for all scalar matrices. Then $\iota_g$ acts on $E(G)$ since, for all $\omega \in E(G)$, 
$$ G^{g(\omega)}=(G^{\omega})^g( \leq G).$$
\noindent Recall that $\Ell(G)=\{G\omega:\omega \in E(G) \}$. The map
$$G\omega \mapsto Gg(\omega)$$
extends naturally to a well-defined action of $Quinn(G)$ on $\Ell(G)$. \\
\noindent (ii) Clearly $G$ acts as M\"obius transformations on $\PP^1(K)$ and it is well-known that 
$$G \backslash \PP^1(K) \leftrightarrow \Cl(A).$$
As we shall see later from the structure of the quotient graph it follows that, for all $k \in \PP^1(K)$, $G(k)$ is infinite, metabelian. Recall that $\Cusp(G)=\{Gk: k \in \PP^1(K)\}$. As before the map $$Gk \mapsto Gg(k)$$ extends to a well-defined action of $Quinn(G)$ on $\Cusp(G)$. \\
\noindent (iii) Serre \cite[Chapter II, Section 1.1, p.67]{Serre} uses \it lattice classes \rm as a model for the vertices and edges of $\TTT$. It is clear that $GL_2(K_{\infty})$ acts naturally on these. In particular the scalar matrices act trivially. The map
$$Gw\mapsto Gg(w),$$
where $w \in \vert(\TTT) \cup \edge(\TTT)$,
extends to a well-defined action of $Quinn(G)$ on the quotient graph $G \backslash \TTT$. Note that $ G_{g(w)}= ((G_w))^g \leq G$. We will use this action to extend a result of Serre. \\
\noindent (iv) Suppose that  
$$\mathcal{S}=\left\{ H \leq G: H \cong GL_2(\FF_q)\right\}$$ 
or $\mathcal{S}$ is a $G$-conjugacy closed subset of  $ \mathcal{C}=\left\{ C \leq G: C \cong \FF_{q^2}^*\right\}$.
\noindent  Then $Quinn(G)$ acts by conjugation on $\mathcal{S}^G$. We use these to define actions of $Quinn(G)$ on significant subsets of $\vert(\TTT)$.

\section{Action on vertex stabilizers}\label{sect:stabilizers}

\noindent Almost all the results in this section hold for all $\delta$. We record the important general properties of subgroups of vertex stabilizers.
\begin{Lem}\label{vertfinite}
(i) $G_v$ is finite, for all $v \in \vert(\TTT)$. \\ \\
(ii) Let $S$ be a finite subgroup of $G$. Then 
$$S \leq G_{v_0},$$ for some $v_0 \in \vert(\TTT)$.

\end{Lem}
\begin{proof} See \cite[Proposition 2, p.76]{Serre}.
\end{proof}
\noindent In this section we are concerned with subgroups of $G_v$ which contain a cyclic subgroup of order $q^2-1$. We record the following result.
\begin{Lem}\label{cyclicsubgroup} Suppose that $G_v$ contains a cyclic subgroup of order $q^2-1$. Then
$$ G_v \cong GL_2(\FF_q)\; or \; G_v \cong \FF_{q^2}^*.$$
\end{Lem}
\begin{proof} See \cite[Corollaries 2.2, 2.4, 2.12]{MSstabilizer}.
\end{proof}
\noindent  In the first part of this section we look at the action of quasi-inner automorphisms on the following set
$$\mathcal{H}=\left\{ H \leq G: H \cong GL_2(\FF_q)\right\}.$$ 
\begin{Lem}\label{vertGL} Let $H \in \mathcal{H}$. Then there exists $v_0 \in \vert(\TTT)$ for which
$$H=G_{v_o}.$$
\end{Lem}
\begin{proof} Follows from Lemmas \ref{vertfinite} (ii) and \ref{cyclicsubgroup}.
\end{proof}
\begin{Rk} (i)  Every $\TTT$ contains a particular vertex $v_s$, usually referred to as \bf standard \rm (after Serre), for which
$$G_{v_s}=GL_2(\FF_q).$$
See \cite[Remark 3), p.97]{Serre} \\
(ii) On the other hand for the case $A=\FF_q[t]$ (equivalently $g(K)=0,\;\delta =1$) it follows from Nagao's Theorem \cite[Corollary, p.87]{Serre} that here $\vert(\TTT)$ has no stabilizer which is cyclic of order $q^2-1$.  
\end{Rk}

\begin{Lem} \label{transGL}
Let $H \in \mathcal{H}$.  Then there exists a quasi-inner automorphism $\kappa=\iota_g$ of $G$ such that 
$$H=\kappa(GL_2(\FF_q)).$$ 
\end{Lem}

\begin{proof}
From the proofs of \cite[Theorem 2.6, Corollary 2.8]{MSstabilizer}, as well as 
\cite[Corollary 2.12]{MSstabilizer} it is clear that there exists 
$$ g= \left[\begin{array}{cc}a&b\\c&d\end{array}\right] \in GL_2(\widetilde{K})$$
\noindent such that $$H=g(GL_2(\FF_q))g^{-1}.$$
\noindent We denote by $\overline{x}$ the image of $x \in \widetilde{K}$ under
 the extension of the Galois automorphism of $\FF_{q^2}/\FF_q$ to $\widetilde{K}$. 
It is clear that $gE_{ij}g^{-1} \in M_2(A)$, where $1 \leq i,j \leq 2$ and so 
$$xy/\Delta(= \overline{xy}/\overline{\Delta}) \in A,$$
\noindent for all $x,y \in \{a,b,c,d\}$, where $\Delta =\det (g)$. 
\\ \\
\noindent Now we may assume without loss of generality that $c \neq 0$. Let $ z \in \{a,b,d\}$. Then
$$c^2/\Delta=\overline{c}^2/\overline{\Delta},\;\;\;cz/\Delta=\overline{cz}/\overline{\Delta}.$$
\noindent It follows that $z/c=\overline{z}/\overline{c}$  so that $z/c \in K$. We now replace $g=M$ with 
$g_0=c^{-1}M$. Then by Theorem \ref{Misquasi}  the map $\kappa_0:G \rightarrow G$ defined 
by $\kappa_0(x)=g_0xg_0^{-1}$  is a quasi-inner automorphism of $G$.   
\end{proof}

\begin{Lem} \label{freeGL}
Let $\kappa_0=\iota_{g_o}$ be a non-trivial quasi-inner automorphism of $G$ and let $H \in \mathcal{H}$. Then
$$\kappa_0(H) \not\sim H.$$ 
\end{Lem}

\begin{proof} By definition $g_0 \in N_{\widehat{G}}(G) \backslash G.Z(K)$. Suppose to the contrary that 
$$\kappa_0(H)=gHg^{-1}$$
\noindent for some $g \in G$. Replacing $g_0$ with $g^{-1}g_0$ we may assume that $g=1$. Now by Lemma \ref{transGL} 
$H=\kappa_0'(GL_2(\FF_q))$ for some quasi-inner $\kappa_0'= \iota_{g_0'}$, say.
\noindent It follows that 
$$g_1(GL_2(\FF_q))g_1^{-1}=GL_2(\FF_q),$$
\noindent where $g_1=(g_0')^{-1}g_0g_0'$. As $N_{\widehat{G}}(G) / G.Z(K)$ is abelian this implies that   
$$ g_1 \equiv g_0\; (\mod Z(K).G)$$ and so we may further assume that $g_1=g_0$. Let
 $$ S_p=\left\{ T(a)=E_{12}(a): a \in \FF_q \right\}.$$
\noindent Now $S_p$ is a Sylow $p$-subgroup of $GL_2(\FF_q)$ and so from the above 
$$g_0(S_p)g_0^{-1} = h(S_p)h^{-1},$$
\noindent for some $h \in GL_2(\FF_q)$. As above we may assume then that $h=1$. It follows that $g_0$  ``fixes" $\infty$ and so 

$$g_0=\left[\begin{array}{cc}\alpha&*\\0&\beta\end{array}\right].$$ 

\noindent By Corollary \ref{M^2} (i) we note that  
$$(\det(g_0)^{-1}\mathrm{tr}((g_0)^2)=\gamma+\gamma^{-1} \in A,$$
\noindent where $\gamma=\alpha\beta^{-1}$. Since $A$ is integrally closed it follows that $\gamma \in A^*(=\FF_q^*)$. 
Then we can replace $g_0$ with $\beta^{-1}g_0$ which belongs to $G$  by Corollary \ref{M^2} (ii). Thus $g_0 \in Z(K).G$.  
\end{proof}

\begin{Lem}\label{stanedge} Let $e \in \edge(\TTT)$ be incident with $v_s$. Then
$$G_e \lneqq GL_2(\FF_q).$$
 	\end{Lem}
 
 \begin{proof} The edges attached to $v_s$ are parametrized by $\PP^1(\FF_{q^{\delta}})$ and $GL_2(\FF_q)$ acts on these as M\"obius transformations. See \cite[Exercise 6), p.99]{Serre}. \\
 If the edge corresponds to $f\in\FF_{q^{\delta}}$ it is not fixed by the translations in $GL_2(\FF_q)$, and if it corresponds to $\infty$, it is not fixed by ${0\ 1\choose 1\ 0}\in GL_2(\FF_q)$.
 \end{proof}

\begin{Prop} \label{noedgeGL}
	No edge of $\TTT$ can have a stabilizer isomorphic to $GL_2(\FF_q)$.
\end{Prop}

\begin{proof} For odd $\delta$ this follows from \cite[Corollary 2.16]{MSstabilizer}. We provide a proof that holds for all $\delta$.  Suppose to the contrary that there is an edge $e$ whose stabilizer is isomorphic to $GL_2(\FF_q)$. Then by Lemma \ref{cyclicsubgroup} the stabilizers of its terminal vertices are both $G_e$.\\ By Lemma \ref{transGL} and the action of quasi-inner automorphisms on $\TTT$ we can assume that 
$$G_e=GL_2(\FF_q).$$ \\
It follows that $GL_2(\FF_q)$ stabilizes the geodesic from $v_s$ to one of the terminal vertices of $e$ which includes $e$ and hence an edge incident with $v_s$. This contradicts Lemma \ref{stanedge}.
\end{proof}
 \begin{Cor}\label{vertstab} 
Let $H \in \mathcal{H}$. Then there exists a unique vertex $v \in \vert (\TTT)$ such that
	$$ G_v =H.$$  
\end{Cor}
\begin{proof} Follows from Lemma \ref{vertGL} and Proposition \ref{noedgeGL}.
\end{proof}
\begin{Rk} Another interesting consequence of Lemma \ref{transGL} and Proposition \ref{noedgeGL} is the following. Suppose that $G_v \in \mathcal{H}$. Then there exists $\kappa=\iota_g$ such that  $\kappa(v)=v_s$
Since $\kappa$ is an automorphism of $\TTT$ the action of $G_v$ on the $q^{\delta}+1$ edges of $\TTT$ incident with $v$ is identical to the action of $GL_2(\FF_q)$ on the edges of $\TTT$ incident with $v_s$ as described in Lemma \ref{stanedge}.
 \end{Rk}
\begin{Def} By definition
$$\vert(G \backslash \TTT)=\left\{Gv: v \in \vert(\TTT) \right\}.$$
We put $\widetilde{v}=Gv$ and define its \bf stabilizer \rm
$$G_{\widetilde{v}}=(G_v)^G.$$
We refer to $G_{\widetilde{v}}$ as being \bf isomorphic to $G_v$.
\end{Def}
 
\begin{Lem}\label{bijection}  There exists a bijection
$$\mathcal{H}^G \longleftrightarrow \left\{ \widetilde{v} \in \vert(G \backslash \TTT): G_v \in \mathcal{H} \right\}.$$
\end{Lem}
\begin{proof} Follows from Corollary \ref{vertstab} and the above.
\end{proof}
	
\noindent It is clear that $Quinn(G)$ acts on $\mathcal{H}^G$. Since $Z(K)$, represented by scalar matrices, acts trivially on $\TTT$, it is also clear that $Quinn(G)$ acts on $G \backslash \TTT$. We now come to the principal result in this section which follows from Lemmas \ref{transGL}, \ref{freeGL} and \ref{bijection}.

\begin{Thm} \label{transandfree}
$Quinn(G)$ acts freely and transitively on 
\begin{itemize}
\item[(i)] the conjugacy classes of subgroups of $G$ which are isomorphic to $GL_2(\FF_q)$,
\item[(ii)] the vertices of $G\backslash \TTT$ whose stabilizers are isomorphic to $GL_2(\FF_q)$ .
\end{itemize}
\end{Thm}
\noindent A special case of this result is provided by Corollary \ref{exquasi-inner}.
\begin{Cor} \label{Corodd} 
Suppose that $|\Cl(A)|$ is odd. Then \begin{itemize}
\item[(i)] every subgroup $H$ of $G$ isomorphic to $GL_2(\FF_q)$
is actually conjugate in $G$ to $GL_2(\FF_q)$,
\item[(ii)] the only vertex in $G\setminus\TTT$ whose stabilizer is isomorphic to $GL_2(\FF_q)$ is 
$\widetilde{v_s}$, the image of the standard vertex $v_s$.
\end{itemize}
\end{Cor}


\section{Action on elliptic points}\label{sect:elliptic}
\noindent Throughout this section we assume that $\delta$ is \it odd. \rm 
\noindent Recall that
$$\Ell(G)=\left\{G\omega: \omega \in E(G) \right\}$$
\noindent denotes the elliptic points of the Drinfeld modular curve $G \backslash \Omega$.
\begin{Def} We define
$$\Ell(G)^{=}=\left\{G\omega: G\omega=G\overline{\omega} \right\}\;\mathrm{and}\;\Ell(G)^{\neq}=\left\{G\omega: G\omega\neq G\overline{\omega} \right\}.$$
\end{Def}
\noindent (In \cite[Section 3]{MSelliptic} $\Ell(G)^=$ is denoted by $\Ell(G)_2$.) \\
\noindent The action of an element of $GL_2(K_{\infty})$ on an element of $\Omega$ will always refer to its action as a M\"obius transformation. We record the following.
\begin{Lem}\label{Ginv} Let $g \in N_{\widehat{G}}(G)$ and $\omega \in E(A)$. Then
\begin{itemize}
\item[(i)] $g(\omega) \in E(A)$,
\item[(ii)] $\overline{g(\omega)}=g(\overline{\omega}).$
\end{itemize}
\end{Lem}
\noindent It is clear then that $Quinn(G)$ acts on both $\Ell(G)^=$ and $\Ell(G)^{\neq}$. \\
\noindent In this section our approach is based on \cite[Sections 3,4]{MSelliptic}. We recall some details.
\begin{Def} Let $I$ be an $A$-ideal (resp. $\widetilde{A}$-ideal). Then $[I]$ denotes the image of $I$ in $\Cl(A)$ (resp. $\Cl(\widetilde{A}))$.
\end{Def}
\noindent Fix $\varepsilon\in\FF_{q^2}\setminus\FF_q$. By \cite[Theorem 2.5]{MSelliptic} any elliptic point $\omega$ of $G$ can be written as $\omega=\frac{\varepsilon+s}{t}$ where $s,t\in A$ and $t$ divides 
$(\overline{\varepsilon} +s)(\varepsilon+s)$ in $A$. Now let
$$J_{\omega}=A(\varepsilon+s)+At.$$ 
\noindent It is known \cite[Lemmas 3.1, 3.2]{MSelliptic} that
\begin{itemize}
\item[(i)] $J_{\omega}$ is an $\widetilde{A}$-ideal.
\item[(ii)] $J_{\omega}$ is independent of the choice of $\varepsilon\in\FF_{q^2}\setminus\FF_q$.
\item[(iii)] Let $\omega, \omega' \in E(A)$. Then 
$$ G\omega=G\omega' \Longleftrightarrow [J_{\omega}]=[J_{\omega'}]\;\;\;(\mathrm{in} \; \Cl(\widetilde{A})).$$
\end{itemize}
\noindent Let $\alpha$ be the Galois automorphism of $\widetilde{K}/K$ (which extends that of $\FF_{q^2}/\FF_q$). Let $k \in \widetilde{K}$. Then the \it norm \rm of $k$ is $k\overline{k}$, where $\overline{k}= \alpha(k)$. Now $\alpha$ restricts to $\widetilde{A}$ and so acts on its ideals and hence its ideal class group. For each $\widetilde{A}$-ideal, $J$, the \it norm \rm of $J$, $N(J)=A \cap(J\bar{J})$, which is an $A$-ideal. We now come to the \it norm map \rm
$$\overline{N} : \Cl(\widetilde{A}) \rightarrow \Cl(A),$$
\noindent where $\overline{N}([I])=[(I\bar{I})\cap A]$. Then
$$[I] \in \mathrm{ker}\; \overline{N} \Longleftrightarrow (I\bar{I})\cap A\; \mathrm{ is\; a\; principal}\; A\mathrm{-ideal}.$$
\noindent We restate \cite[Theorem 3.4]{MSelliptic}.
\begin{Thm}\label{norm} The map $\omega \mapsto [J_{\omega}] $ induces a one-one correspondence 
$$\Ell(G) \longleftrightarrow \mathrm{ker} \; \overline{N}.$$
\end{Thm}
\noindent For each $\omega$ it is known that
\begin{itemize}
\item[(i)] $\overline{J_{\omega}}=J_{\overline{\omega}}$,
\item[(ii)] $J_{\omega}J_{\overline{\omega}}$ is a principal $A$-ideal.
\end{itemize}
\noindent It follows that 
$$\mathrm{ker} \; \overline{N}=\left\{ [J_{\omega}]: [J_{\overline{\omega}}]=[J_{\omega}]^{-1} \right\}.$$

\noindent We recall from Theorem \ref{isoto2tors} that $Quinn(G)$ can be identified with $\Cl(A)_2$. From this and  Theorem \ref{norm} we are able to study the action of $Quinn(G)$ on $\Ell(G)$. For this purpose we require two further lemmas.
\begin{Lem}\label{canonical} Let $\iota: \Cl(A) \rightarrow \Cl(\widetilde{A})$ be the canonical map, where
$\iota([I])=[I\widetilde{A}],\;(I\unlhd A).$ Then
\begin{itemize}
\item[(i)] $\iota$ is injective.
\item[(ii)] $\left\{ [I] \in \Cl(\widetilde{A}); [I]=[\overline{I}] \right\}= \iota (\Cl(A))$. \end{itemize}
\end{Lem}
\begin{proof} The analagous statements are known to hold for the canonical map from $  \Cl^0(K) \to \Cl^0(\widetilde{K})$. See \cite[Corollary to Proposition 11.10]{Rosen}. The results follow from the exact sequence in Section 2, since $\delta$ is odd and the infinite place is inert in $\widetilde{K}$.
\end{proof}

\begin{Lem}\label{2torsion} With the above notation, the $2$-torsion in $\Cl(A)$,
$$\Cl(A)_2 \cong \iota(\Cl(A)_2)= (\mathrm{ker}\;\overline{N})_2,$$
the $2$-torsion in $\mathrm{ker}\;\overline{N}$.
\end{Lem}
\begin{proof} Let $[I] \in \Cl(A)_2$. Then $\iota([I])$ has order $2$ in $\Cl(\widetilde{A})$ by Lemma \ref{canonical}. Now
$$ \overline{N}(\iota([I]))=\iota([I])\overline{\iota([I])}=(\iota([I]^2)=1,$$
by Lemma \ref{canonical} (ii). Hence $\iota([I]) \in \mathrm{ker}\; \overline{N}$.
Conversely let $[J] \in Cl(\widetilde{A})$ have order $2$ and lie in $\mathrm{ker} \; \overline{N}$. 
Then $[J]^2=1$ and $[J][\overline{J}] = \overline{N}([J])=1$. Hence $[J]=[\overline{J}]$ and so $[J] \in\iota( \Cl(A)_2)$ again by Lemma \ref{canonical} (ii).
\end{proof}
\noindent Any element of $N_{\widehat{G}}(G)$ can be represented by a matrix
$$ M= \left[\begin{array}{cc}a&b\\c&d
\end{array}\right] \in \widehat{G}.$$
\noindent By multiplying $M$ by a suitable scalar matrix we may assume that $a,b,c,d \in A$. As before let 
$$\qq(M):=(a)+(b)+(c)+(d).$$
Then
\begin{itemize}
\item[(i)] $\qq(M)^2=(\varDelta).$
\item[(ii)] $(a)+(b)=(a)+(c)=(d)+(b)=(d)+(c)=\qq(M).$
\end{itemize}
\noindent See Theorem \ref{Misquasi} and \cite[Remarks $2$]{Cremona}. In this way $\qq$ induces an isomorphism from $Quinn(G)$ onto $\Cl(A)_2$ and so $\iota\circ \qq$ provides an embedding of $Quinn(G)$ into $\Cl(\widetilde{A})$.
\noindent As before each $\omega \in E(G)$ can be represented as $\omega=\frac{\varepsilon+s}{t}$ where $s,t\in A$ and $t$ divides $(\varepsilon^q +s)(\varepsilon+s)$ in $A$. The element $M$ acts as a M\"obius transformation on $\omega$ by multiplying the column vector ${\varepsilon+s\choose t}$ on the left by the matrix $M$. It follows that $J_{M(\omega)}$ is the $\widetilde{A}$-ideal generated by $a(\varepsilon+s)+bt$ and $c(\varepsilon+s)+dt$. Our next result, the most important in this section, shows that the action of $Quinn(G)$ on $\Ell(G)$ is equivalent to group multiplication in $\mathrm{ker}\;\overline{N}$.
 
\begin{Thm} \label{multiplicationelli} With the above notation,

$$[J_{M(\omega)}]=[\iota(\qq(M))J_{\omega}]=[\iota(\qq(M))][J_{\omega}]\;\;\; (\mathrm{in} \; \ker\; \overline{N}).$$
\end{Thm}

\begin{proof} From the above it is clear that $J_{M(\omega)} \leq \qq(M)J_{\omega}$. Since $\widetilde{A}$ is a Dedekind domain, there is an integral ideal $I_1$ of $\widetilde{A}$ such that
$$J_{M(\omega)}=\qq(M)J_{\omega}I_1.$$
\noindent By the same argument there exists an integral ideal $I_2$ of $\widetilde{A}$ with 
$$J_{M^2 (\omega)}=\qq(M)J_{M(\omega)}I_2 =\qq(M)^2 J_{\omega}I_1 I_2 = \varDelta J_{\omega}I_1 I_2.$$
On the other hand, from part (i) of Corollary \ref{M^2} we see that $J_{M^2 (\omega)}=\varDelta J_{\omega}$.
Hence $I_1 = I_2 =\widetilde{A}$ and the result follows.
\end{proof}
\noindent An immediate consequence is the following.
\begin{Cor} \label{freeelli}
$Quinn(G)$ acts freely on $Ell(G)$. More precisely, a quasi-inner automorphism that fixes
an elliptic point in $G\setminus\Omega$ must necessarily be inner.
\end{Cor}
\noindent Since $$G\omega=G\overline{\omega} \Longleftrightarrow [J_{\overline{\omega}}]=[J_{\omega}]=
[J_{\omega}]^{-1}$$
\noindent we can identify $\Ell(G)^=$  with $\iota(\Cl(A)_2)\cong Quinn(G))$. Combining Lemma \ref{2torsion} and Corollary \ref{freeelli} we obtain the following result.
\begin{Thm}\label{transonelli} $Quinn(G)$ acts freely and transitively on $\Ell(G)^=$.
\end{Thm}
\noindent Theorem \ref{transandfree} (ii), which holds for all $\delta$,
provides an alternative proof of Theorem \ref{transonelli}. Applying the former for the case of odd $\delta$ the latter then follows from the existence of a $Quinn(G)$-invariant one-one correspondence between $\Ell(G)^=$ and 
$\left\{ \widetilde{v} \in \vert(G \backslash \TTT): G_v \cong GL_2(\FF_q) \right\}$. \\ \\
\noindent From the above it is clear that $|\Ell(G)|=n_E|\Ell(G)^=|$, where $$n_E=|\ker\; \overline{N}:\iota(\Cl(A)_2)|.$$ It follows that 
$|\Ell(G)^{\neq}|=(n_E-1)|\Ell(G)^=|$. \\ 

\noindent We recall that the \it building map \rm \cite[p.41]{Gekeler} restricts to a map
$$\lambda: E(G) \rightarrow \vert(\TTT),$$
for which $G^{\omega} \leq G_{\lambda(\omega)}$.
\noindent Let $\kappa$ be a quasi-inner automorphism. Then  by \cite[(iii), p.44]{Gekeler}
$$\lambda(\kappa(\omega))=\kappa(\lambda(\omega)).$$
\noindent Then $\lambda$ induces a map
$$\Ell(G) \mapsto \vert{\TTT}.$$
\noindent By Lemma \ref{Ginv} (ii), Theorem \ref{transandfree} and \cite[Proposition 3.4]{MSelliptic}  this leads to two $Quinn(G)$-invariant one-one correspondences.
$$\Ell(G)^= \longleftrightarrow \left\{\tilde{v} \in \vert(G \backslash \TTT): G_v \cong GL_2(\FF_q) \right\},$$
$$\mathcal{G}= \left\{ \{G\omega,G\overline{\omega}\}:G\omega\neq G\overline{\omega} \right\} \longleftrightarrow \mathcal{V}=\left\{\tilde{v} \in \vert(G \backslash \TTT): G_v \cong \FF_{q^2}^* \right\}. $$ \\ 
\noindent  Note that $|\mathcal{G}|=\frac{1}{2}|\Ell(G)^{\neq}|$.
\begin{Lem}\label{order4} Let $G\omega \in \Ell(G)^{\neq}$ and $\kappa$ be a quasi-inner automorphism represented by $M\in N_{\widehat{G}}(G)$. Then 
$\kappa(G\omega)=G\overline{\omega}$ if and only if $[J_{\omega}]$ has order $4$ in $\mathrm{ker}\;\overline{N}$ and $\iota(\qq(M))]=[J_{\omega}]^2$.
\end{Lem} 
\begin{proof} Let $n>2$ be the order of $[J_{\omega}]$ in $\mathrm{ker}\; \overline{N}$. If $\kappa(G\omega)=G\overline{\omega}$ then by Theorem \ref{multiplicationelli}
$$[\iota(\qq(M)][J_{\omega}]=[J_{\omega}]^{n-1}.$$
Hence $[\iota(\qq(M)]=[J_{\omega}]^{n-2}=[J_{\omega}]^{-2}$ and so $n=4$. The converse is straightforward.
\end{proof}
\noindent The following is an immediate consequence.
\begin{Lem}\label{orbit} Let $\widetilde{v} \in \mathcal{V}$ and let $\{G\omega,G{\overline{\omega}}\}$ be the corresponding elliptic element of $\widetilde{v}$. Then the length of the orbit of $\widetilde{v}$ under the action of $Quinn(G)$ is $\frac{1}{2}|Quinn(G)|$ if $[J_{\omega}]$ has order $4$ in $\mathrm{ker}\;\overline{N}$ and $|Quinn(G)|$ otherwise.
\end{Lem}
\begin{Prop} \label{2,3,odd}
Suppose that $|\Ell(G)^=|<|\Ell(G)|$. Then
\begin{itemize} 
	\item[(a)] $Quinn(G)$ acts transitively on $\Ell(G)^{\neq}$ if and only if $n_E=2$.
	\item[(b)] $Quinn(G)$ acts transitively on $\mathcal{V}$ if and only if $n_E\in\{2,3\}$.
           \item[(c)] $Quinn(G)$ acts freely on $\mathcal{V}$ if and only if $n_E$ is odd.
           \item[(d)] $Quinn(G)$ acts freely and transitively on $\mathcal{V}$ if and only if $n_E=3$.
\end{itemize} 
\end{Prop}

\begin{proof} 
(a) Since $Quinn(G)$ acts freely on $\Ell(G)^{\neq}$  the action is transitive if and only if 
$|Quinn(G)|=|\Ell(G)^=|= |\Ell(G)^{\neq}|$ that is if $n_E=2$. \\
(b) If $Quinn(G)$ acts transitively on $\mathcal{V}$ then $|\mathcal{G}| \leq |\Ell(G)^=|$ and so $n_E \in \{2,3\}$. When $n_E=2$ (a) applies. When $n_E=3$ the two $Quinn(G)$-orbits represented by $G\omega$ and $G\overline{\omega}$ are identified in $\mathcal{G}$. \\
(c) By Lemma \ref{order4} the action of $Quinn(G)$ on $\mathcal{G}$ is \it not \rm free if and only if there exists $[J_{\omega}]$ of order $4$ and such an element exists if and only if $n_E$ is even.\\
(d) follows from (b) and (c).
\end{proof}
\begin{Rk} Suppose that $g(K)=g>0$. The $2$-torsion rank of an abelian variety of dimension $g$ is bounded 
by $2g$. Applying this to $\Cl^0(\widetilde{K})$ or $\Cl(\widetilde{A})$ (and using the fact that $\delta$ is odd) it follows that 
$$|\Ell(G)^=| \leq 2^{2g}.$$
See \cite[Chapter 11]{Rosen}. On the other hand by the Riemann Hypothesis for function fields 
\cite[Theorems 5.1.15(e), 5.2.1]{Stich}
$$|\Ell(G)|=L_{K}(-1) \geq (\sqrt{q}-1)^{2g}.$$
If $n_E=2$ then
$$2^{2g+1} \geq (\sqrt{q}-1)^{2g}.$$
(a)  \it If $q\geq 16$ (and $g>0$), then $Quinn(G)$ cannot act transitively on $\Ell(G)^{\neq}$. \\ 
\rm Another consequence follows using an identical argument. \\
(b) \it If $q\geq 23$ (and $g >0$), then $Quinn(G)$ cannot act transitively on $\mathcal{V}$. 	
\end{Rk} 
\begin{Rk} It is known \cite[Corollary 2.12, Theorem 5.1]{MSstabilizer} that
a vertex $\widetilde{v}$ of $G \backslash \TTT$ is \bf isolated \rm if and only if $\delta=1$ and $G_v \cong GL_2(\FF_q)$ or $\FF_{q^2}^*$. Hence when $\delta=1$ therefore  Theorem \ref{transonelli}, Proposition \ref{2,3,odd} and Remarks 4.13 can be interpreted as statements about the action of $Quinn(G)$ on the isolated vertices of $G \backslash \TTT$.
\end{Rk}
\section{Action on cyclic subgroups}\label{sect:cyclic}
\noindent Our focus of attention in this section are the subgroups of $G$ which are cyclic of order $q^2-1$. As distinct from Section 3 some of the results require $\delta$ to be \it odd. \rm
\begin{Def} A finite subgroup $S$ of $G$ is \bf maximally finite \rm if every subgroup of $G$ which properly contains it is infinite.
\end{Def}

\begin{Lem}\label{contnotmaxfinite} Let $C$ be a cyclic subgroup of $G$ of order $q^2-1$ which is not maximally finite. Then there exists $H \in \mathcal{H}$ which contains $C$. Moreover $H$ is unique if $\delta$ is odd.
\end{Lem}
\begin{proof} By Lemma \ref{vertfinite} (ii) there exists $G_v$ which properly contains $C$. Hence $G_v \in \mathcal{H}$ by Lemma \ref{cyclicsubgroup}. \\
\noindent Suppose now that $\delta$ is odd. If $H$ is not unique then
$$ C \leq G_{v_1} \cap G_{v_2},$$
where $v_1 \neq v_2$. It follows that $C$ fixes the geodesic in $\TTT$ joining $v_1$ and $v_2$, including all its edges. This contradicts \cite[Corollary 2.16]{MSstabilizer}.  
\end{proof}

\begin{Lem}\label{cyclicconjugate} Let $C,C_0$ be cyclic subgroups of order $q^2-1$ contained in some $H \in \mathcal{H}$. Then $C,C_0$ are conjugate in $H$.
\end{Lem}
\begin{proof} By Lemma \ref{transGL} we may assume that $H=GL_2(\FF_q)$. This then becomes a well-known result. In the absence of a suitable reference we sketch a proof which lies within the context of this paper. \\
By the proof of \cite[Theorem 2.6]{MSstabilizer} (based on \cite[Lemma 1.4]{MSstabilizer}) it follows that
$$C=F^{\mu}=\left\{ g \in GL_2(\FF_q) : g(\mu)=\mu \right\},$$
for some $\mu \in \FF_{q^2} \backslash \FF_q$. Let $C_0=F^{\mu_0}$.\\
Now $\mu_0=\alpha \mu+\beta$ for some $\alpha,\beta \in \FF_q$, where $\alpha \neq 0$. Then
$C_0= g_0Cg_0^{-1}$, where

$$ g_0= \left[\begin{array}{cc}\alpha&\beta\\0&1\end{array}\right].$$
\end{proof}
\begin{Def} Let
$$\mathcal{C}=\left\{C\leq G:C, \mathrm{cyclic \;of\;order} \; q^2-1 \right\},$$
$$\mathcal{C}_{mf} = \left\{ C \in \mathcal{C}: C, \mathrm{maximally \;finite} \right\}.$$
 $$ \mathcal{C}_{nm}= \mathcal{C} \backslash \mathcal{C}_{mf}.$$
\end{Def}
\noindent Clearly every automorphism of $G$ acts on both $\mathcal{C}_{mf}$ and $\mathcal{C}_{nm}$.

\begin{Prop}\label{transcyclic}	The quasi-inner automorphisms act transitively on all cyclic subgroups of $G$ of order $q^2-1$ that are not maximally finite.
\end{Prop}
\begin{proof} Let $C \in \mathcal{C}_{nm}$. Then by Lemmas \ref{transGL} and \ref{contnotmaxfinite} there
exists $g_0 \in N_{\widehat{G}}(G)$ such that   
$$ C^{g_0} \in GL_2(\FF_q).$$
The rest follows from Lemma \ref{cyclicconjugate}.
\end{proof}

\noindent The next result follows from Proposition \ref{transcyclic} and Theorem \ref{transandfree}. 
\begin{Prop} \label{freetranscyclic} 
	If $\delta$ is odd, $Quinn(G)$ acts freely and transitively on the conjugacy classes (in $G$) of cyclic subgroups of $G$ of order $q^2-1$ that are not maximally finite.
	\end{Prop}
\noindent The restrictions on $\delta$ in Lemma \ref{contnotmaxfinite} and Proposition \ref{freetranscyclic} are necessary.
\begin{Exm} \label{exg=0delta=2} Consider the case where $g(K)=0, \;\delta=2$. This case is studied in detail in \cite[Section 3]{MSJLMS}. By the exact sequence in Section 2 it is known that here
$$ \Cl(A)=\Cl(A)_2 \cong Quinn(G) \cong \ZZ/2\ZZ.$$
\noindent There exists a vertex $v_0$ adjacent to the standard vertex $v_s$ and $g_0 \in N_{\widehat{G}}(G) \backslash G$ such that 
$$G_{v_0}=GL_2(\FF_q)^{g_0}\; \mathrm{and}\; G_{v_s} \cap G_{v_0} \in \mathcal{C}_{nm}.$$
\noindent Hence the restriction on $\delta$ in part of Lemma \ref{contnotmaxfinite} is necessary. \\ \\
\noindent It is known \cite[Theorem 3.3]{MSJLMS} that in this case
$$G= GL_2(\FF_q) \star_{\quad{C}} GL_2(\FF_q)^{g_0},$$
where $C(=GL_2(\FF_q) \cap GL_2(\FF_q)^{g_0}) \in \mathcal{C}_{nm}$. It follows by Lemma \ref{cyclicconjugate} that
there exists $g \in GL_2(\FF_q)$ for which
 $$ C^g=C^{g_0}.$$
 \noindent In this case therefore $Quinn(G)$, which is non-trivial, fixes $C^G$. The restriction on $\delta$ in 
Proposition \ref{freetranscyclic} is therefore necessary.
\end{Exm}
\noindent We conclude this section with some remarks about $\mathcal{C}_{mf}$. 
\begin{Lem}\label{maxfiniteodd} Suppose that $\delta$ is odd. Then
$$C\in \mathcal{C}_{mf} \Longleftrightarrow C=G_v \cong \FF_{q^2}^*.$$
\end{Lem}
\begin{proof} Suppose that $C=G_v \cong \FF_{q^2}^*$ and that $C \in \mathcal{C}_{nm}$. Then by Lemmas \ref{vertfinite} and \ref{vertGL} it follows that $C \leq G_v \cap G_{v_0}$ for some $v_0 \neq v$, which contradicts 
\cite[Corollary 2.16]{MSstabilizer}. The rest follows from Lemma \ref{vertfinite}.
\end{proof}
\noindent When $\delta$ is odd there is therefore a one-one correspondence
$$(\mathcal{C}_{mf})^G \longleftrightarrow \mathcal{V}.$$
\noindent For the case where $\delta$ is odd this shows that the results in Proposition \ref{2,3,odd} apply to the action of $Quinn(G)$ on $(\mathcal{C}_{mf})^G$.

\begin{Rk} \label{maxcycliceven} 
As a M\"obius transformation every member of $G$ fixes an element of $C_{\infty}$. Suppose now that $\delta$ is even and that $C$ is a cyclic subgroup of order $q^2 -1$ (maximally finite or not). Then from the proof of \cite[Proposition 2.3]{MSelliptic} it follows that $C$ fixes $\mu \in K.\FF_{q^2}\backslash K$. In this case however $\mu \in K_{\infty}$ as $\delta$ is even. So $\mu$, which is not in $\Omega$ and not in $K$, can neither be an inner point nor a cusp of the Drinfeld modular curve $G \backslash \Omega$. We refer to $\mu$ as \bf pseudo-elliptic. \rm \\
\noindent On the other hand suppose that $\delta$ is odd. Let $g$ be any element of infinite order in $G$ and let $g$ fix $\lambda$. Then $\lambda \in K_{\infty} \backslash K$. \\
 
\end{Rk}

\section{Action on cusps}\label{sect:cusps}
\noindent As distinct from Section $4$ the results here hold \it for all $\delta$. \rm Any element of $\widehat{G}$ acts on $\PP^1(K)=K \cup \{\infty\}$ as a M\"obius transformation. In this way $Quinn(G)$ acts on $G \backslash \PP^1(K)=\Cusp(G)$. Every element of $\Cusp(G)$ can be represented in the form $(a : b)$, where $a,b \in A$. Since $A$ is a Dedekind ring this gives rise to a one-one correspondence
$$\Cusp(G) \longleftrightarrow \Cl(A).$$ Hence the action of $Quinn(G)$ on $\Cusp(G)$ translates to an action of $\Cl(A)_2$ on $\Cl(A)$. The principal result in this section is similar to but simpler than Theorem \ref{multiplicationelli}. It translates this action into multiplication in the group $\Cl(A)$. We sketch a proof. \\
\noindent We can represent any cusp, $c$, by an element $(x:y) \in \PP^1(K)$, where $x,y \in A$. Let
$$J_c= xA+yA,$$
\noindent and let $[J_c]$ be its image in $\Cl(A)$. \\
\noindent Now let $\kappa$ be a non-trivial element of $Quinn(G)$. Then as before by Theorem \ref{Misquasi}  $\kappa$ can be represented by a matrix 
$$ M= \left[\begin{array}{cc}a&b\\c&d
\end{array}\right] \in \widehat{G},$$
\noindent where we may assume that $a,b,c,d \in A$.  Let $\qq(M)$ be the $A$-ideal generated by
$a,b,c,d$.\\
\noindent The action of $\kappa$ on $c$ is given by the action of $M$ multiplying the column vector ${x \choose y}$ on the left by $M$. In this way
$$ J_{\kappa(c)}=J_{M(c)}=(ax+by)A+(cx+dy)A.$$

\begin{Thm} \label{multiplicationcusp} Under the identification of $\Cusp(G)$ with $\Cl(A)$ and $Quinn(G)$ with $\Cl(A)_2$ the action of $Quinn(G)$ on the cusps translates into multiplication in the group $\Cl(A)$. More precisely

$$[J_{\kappa(c)}]=[\qq(M)J_c]=[\qq(M)][J_c]\;\;\;(in \;\Cl(A)).$$ 
\end{Thm}
\begin{proof} Since $A$ is a Dedekind domain there exists an $A$-ideal $I_1$ such that
$$J_{M(c)}=\qq(M)J_{c}I_1.$$
\noindent By Corollary \ref{M^2} (i) there exists an $A$-ideal $I_2$ with 
$$\Delta J_{c}=J_{M^2 (c)}=\qq(M)J_{M(c)}I_2 =\qq(M)^2 J_{c}I_1 I_2 = \Delta J_{c}I_1 I_2,$$
where $\Delta=\det(M)$.
Hence $I_1 = I_2 =A$ and the result follows.
\end{proof}

As in the previous section we have the following immediate consequence.

\begin{Cor} \label{actiononcusps}
If a non-trivial quasi-inner automorphism $\kappa$ fixes any cusp, then $\kappa$ reduces to an inner automorphism.
In particular, $Quinn(G)$ acts freely on $\Cusp(G)$.
\end{Cor}

\begin{Rk} \label{transoncusps}
\noindent  \rm $Quinn(G)$ acts transitively on $\Cusp(G)$ if and only if $\Cl(A)_2 =\Cl(A)$. \\
From the exact sequence in Section $2$  a necessary condition for this is $\delta\in\{1,2\}$. If $g(K)=0$, this condition is also sufficient,  as then $\mathrm{Cl}(A)\cong\ZZ/\delta\ZZ$.

\noindent But if $g(K)=g>0$, the action cannot be transitive for $q>9$ by an argument very similar to that used in  Remark 4.11. The inequality  
$$\frac{|\mathrm{Cl}^0(K)|}{|\mathrm{Cl}^0(K)_2|}\geq\frac{(\sqrt{q}-1)^{2g}}{2^{2g}}$$
shows that for fixed $q>9$ the number of orbits of $Quinn(G)$ on $\Cusp(G)$ tends to $\infty$ with $g(K)$. \\
\end{Rk}
\noindent The cusp $\infty$($={1\choose 0}$) corresponds to the principal $A$-ideals. Its orbit under $Quinn(G)$ corresponds to the $2$-torsion in $\Cl(A)$ and in the sense of Theorem \ref{multiplicationcusp} the action of $Quinn(G)$ on it translates into $\Cl(A)_2$ acting on itself by multiplication. \\
\noindent For every cusp $c$ represented by the ideal class $[J_c]$ in $\Cl(A)$ there corresponds its (group) inverse $[J_c]^{-1}$ in $\Cl(A)$. We can partition $\Cl(A)$ thus
$$Quinn(G) \leftrightarrow \Cl(A)_2 = \{[J_c]:[J_c]=[J_c]^{-1}\},$$
$$\Cl(A) \backslash \Cl(A)_2 = \{[J_c]:[J_c]\neq[J_c]^{-1}\}.$$

\noindent Our next result follows from Theorem \ref{multiplicationcusp}  analogous to the way Lemma \ref{order4} follows from Theorem \ref{multiplicationelli}.
\begin{Lem}\label{order4cusp} A quasi-inner automorphism $\kappa$, represented by $ M \in N_{\widehat{G}}(G)$, maps the cusp $c$ corresponding to $[J_c]$ in
$\Cl(A) \backslash \Cl(A)_2$, to the cusp corresponding to $[J_c]^{-1}$
if and only if $[J_c]$ has order $4$ and $[J_c]^2=\qq(M)$.
\end{Lem}
\noindent In the next section we will use the results in Sections $5$ and $6$, together with 
Theorem \ref{transandfree} (ii), to examine in detail the action of $Quinn(G)$ on $G \backslash \TTT$. 
\section{Action on the quotient graph}\label{sect:graph}

\noindent The model used by Serre for $\TTT$ \cite[Chapter II, Section 1.1]{Serre} is based on two-dimensional so called \it lattice classes \rm. Since every quasi-inner automorphism, $\iota_g$,  can be represented by a matrix in $\widehat{G}$ it acts on $\TTT$ and hence $Quinn(G)$ acts on $G\backslash \TTT$.

\noindent In this section we investigate the action of a quasi-inner automorphism on the quotient graph $H \backslash \TTT$, where $H$ is a finite index subgroup of $G$. In the process we extend a  result  of Serre \cite[Exercise 2(e), p.117]{Serre} which motivated our interest in this question. We begin with a detailed account of Serre's classical  description of $G \backslash \TTT$. Serre's original proof  \cite[Theorem 9, p.106]{Serre} is based on the theory of vector bundles. For a more detailed version which refers explicitly to matrices see \cite{MasonTAMS}. In addition we use the results the previous sections to shed new light on the structure of $G \backslash \TTT$.\\

\begin{Def} A \bf ray \rm $\mathcal{R}$ in a graph $\mathcal{G}$ is an infinite half-line, without backtracking. In accordance with Serre's terminology \cite[p.104]{Serre} we call $\mathcal{R}$ \bf cuspidal \rm if all its non-terminal vertices have valency $2$ (in $\mathcal{G}$). 
\end{Def}

\noindent  Let $\{g_1, \cdots ,g_s\} \subseteq \widehat{G}$, where $s \geq 1$, be a complete system of representatives for $\Cl(A)_2( \cong N_{\widehat{G}}(G)/G.Z(K))$. Let $c_i=g_i(\infty)\;(1 \leq i \leq s)$. We will assume that $c_1=\infty$.  If $\Cl(A)=\Cl(A)_2$, then $\{c_1, \cdots , c_s\}$ is a complete system of representatives for $\Cl(A)$. If $\Cl(A) \neq \Cl(A)_2$ we can find further elements $h_1, \cdots ,h_t \in \widehat{G}$, where $t \geq 1$ so that
$$\mathcal{S}=\{ c_1, \cdots, c_s,\;d_1, \cdots,d_t\}$$ is a complete set of representatives for $\Cl(A)$, where $d_j=h_j(\infty)\;\;(1 \leq j \leq t)$. \\ \\
\begin{Thm}\label{Serre}  There exists a
complete system of representatives $\mathcal{C}\;(\subseteq \PP^1(K))$ for $\Cusp(G)$(equivalently, $\Cl(A))$ of the above type
 such that
$$ G \backslash \TTT=X \cup\left(\bigcup_{1 \leq i \leq s}\mathcal{R}(c_i)\right)\left(\bigcup_{1 \leq j \leq t} \mathcal{R}(d_j)\right),$$
\noindent  where
\begin{itemize}
\item[(i)] $X$ is finite,
\item[(ii)] each $ \mathcal{R}(c_i),\; \mathcal{R}(d_j)$ is a cuspidal ray (in $G\backslash \TTT$),
whose only interesection with $X$ consists of a single vertex,
\item[(iii)] the $|\Cl(A)|$ cuspidal rays are pairwise disjoint.
\end{itemize}
\noindent Moreover if $\mathcal{R}(e)$ is any of these cuspidal rays then it has a lift, $\overline{\mathcal{R}(e)}$, to $\TTT$
with the following properties. Let $\vert(\overline{\mathcal{R}(c)})=\{v_1, v_2 \cdots \}$. Then
\begin{itemize}
\item[(i)] $G_{v_i} \leq G_{v_{i+1}},\; (i \geq 1),$
\item[(ii)] $$\bigcup_{i \geq 1}G_{v_i}=G(c),$$
where $G(c)$ is the stabilizer (in $G$) of the cusp $c$.
\end{itemize}
\end{Thm}
\noindent \rm  For each $j$ let $\widetilde{d_j}$ be the element of $\{d_1, \cdots,d_t\}$ corresponding to $h_j^{-1}(\infty)$. We may relabel the latter set as $\{d_1,\widetilde{d_1},\cdots,d_{t'},\widetilde{d_{t'}}\}$, where $t'=t/2$. We can use the results in Section 3 to elaborate on the structure of the above cuspidal rays.
We recall that $$ \mathcal{H}=\{ H \leq G: H \cong GL_2(\FF_q) \}.$$
\begin{Cor}\label{Serrecor} For the above set of $|\Cl(A)|$ cuspidal rays\begin{itemize}
\item[(i)] $$\mathcal{R}_1=\{\mathcal{R}(c_1), \cdots ,\mathcal{R}(c_s)\} \leftrightarrow \{\widetilde{v} \in \vert(G \backslash \TTT): G_v \in \mathcal{H}\}\leftrightarrow \Cl(A)_2.$$
\item[(ii)] $$\mathcal{R}_2=\left\{ \mathcal{R}(d_j),\mathcal{R}(\tilde{d_j}): 1\leq j \leq t' \right\}\leftrightarrow \Cl(A) \backslash\Cl(A)_2.$$
\end{itemize}
\end{Cor}

\begin{proof} Let $\widetilde{v} \in \vert(G \backslash \TTT)$, where $G_v \in \mathcal{H}$, and let $H \in \mathcal{H}$ be any representative of its stabilizer. Then, for some {\it unique} $i$,
$$H=gg_i(GL_2(\FF_q))(gg_i)^{-1},$$
\noindent where $g \in G$, by Lemmas \ref{transGL} and \ref{freeGL}. Now let $u$ be any unipotent element of $H$. Then $u$ fixes $gg_ih(\infty)$, for some $h \in GL_2(\FF_q)$. It follows that
$$u \in G(c) \Longleftrightarrow c=g'c_i,$$
\noindent where $g' \in G$. The rest follows from Corollary \ref{vertstab} together with Theorem \ref{transandfree}.
\end{proof}
\begin{Rk} Let $\widetilde{v} \in \vert(G \backslash \TTT)$, where $G_v \in\mathcal{H}$. Then it is shown in Corollary \ref{Serrecor} that $\widetilde{v}$ is adjacent in $G \backslash \TTT$
to a vertex whose stabilizer (up to conjugacy in $G$) is contained in $G(c_i)$, for some {\it unique} $i$. In this way $\widetilde{v}$ can be thought of as {\it closer} in $G \backslash \TTT$ to $\mathcal{R}(c_i)$ than to any other cuspidal ray. For the case $\delta=1$ (and only for this case) $\widetilde{v}$ is {\it isolated} in $G \backslash \TTT$ by \cite[ Theorem 5.1]{MSstabilizer}. As in Takahashi's example [Ta] such a $\widetilde{v}$ then appears as a ``spike" next to its associated cuspidal ray.

\end{Rk}
\noindent  \rm  For each subgroup $H$ of $G$ we recall that the elements of $H \backslash \TTT$ are
$$\vert(H \backslash \TTT)=\left\{Hv: v \in \vert(\TTT)\right\} \;\;\mathrm{and} \;\;\edge(H\backslash \TTT)=\left\{He: e \in \edge(\TTT)\right\}.$$ 

\begin{Def} Let $H,H^*$ be isomorphic subgroups of $G$. An isomorphism of graphs
$$\phi: H \backslash \TTT \rightarrow H^* \backslash \TTT,$$
\noindent is said to be \bf stabilizer invariant \rm if the following condition holds. 
\\ \\
\noindent For any $ w \in \vert(\TTT) \cup \edge(\TTT)$ let
$$\phi(Hw)=H^*w^*,$$
\noindent (where $w^* \in \vert(\TTT)$ if and only if $w \in \vert(\TTT)$). Then, for all $u \in H_w$ and $u^* \in H^*w^*$
$$H_u \cong H_{u^*}^*.$$ 
\end{Def}

\noindent As we shall see it is easy to find examples of isomorphisms of quotient graphs which are not stabilizer invariant.
\begin{Thm}\label{quasiquot} Let $\kappa=\iota_g$, where $g \in N_{\widehat{G}}(G)$ and let $H$ be a subgroup of $G$. Then the map
$$\overline{\kappa}_H: H \backslash \TTT \rightarrow \kappa(H) \backslash \TTT,$$
\noindent defined by
$$ \overline{\kappa}_H(Hw)= H'w',$$
\noindent where $H'=H^g=gHg^{-1}$, $w'=g(w)$ and $w \in \vert(\TTT) \cup \edge(\TTT)$, defines a stabilizer invariant  isomorphism of the quotient graphs
$$ \kappa(H)\backslash \TTT \cong H \backslash \TTT.$$ 
\end{Thm}

\begin{proof} Note that $\overline{\kappa}_H$ is well defined since if $\kappa(x)=g_1xg_1^{-1}$, where $g_1 \in N_{\widehat{G}}(G)$ then $gg_1^{-1} \in Z(K)$ and $Z_{\infty}$, the set of scalar matrrices in $GL_2(K_{\infty})$, stabilizes every $w$. The rest is obvious (since $g$ acts on $\TTT$). 
\end{proof}
\noindent  \rm Let $H$ be any finite index subgroup of $G$ and let $M$ be the largest normal subgroup of $G$ contained in $H$. Then $N =M \cap M^g$ is the largest (finite index) subgroup of $G$, contained in $H$, which is normalized by $G,\;Z(K)$ and $g$. (See Section $2$.) 
\begin{Cor}\label{quasiquotcor} Suppose that $\kappa$ is non-trivial (i.e. $g \notin G.Z(K)$). Let $N$ be a finite index normal subgroup of $G$  normalized by $\kappa$. Then the map
$$\overline{\kappa}_N: N \backslash \TTT \rightarrow N \backslash\TTT,$$
\noindent defined as above, is a non-trivial stabilizer invariant automorphism whose order $n$ is even. Moreover, if $Z \leq N$, then $n=2m$, where $m$ divides $|G:N|$. 
\end{Cor}
\begin{proof} To prove that $\overline{\kappa}_N$ is non-trivial it suffices to prove that $\overline{\kappa}_G$ is not the identity map. There exists $v_0 \in \vert(\TTT)$ for which (non-central) $G_{v_0} \leq G(\infty)$ \cite[Lemma 3.2]{MasonTAMS}. Suppose to the contrary that
$\overline{\kappa}_G$ fixes $Gv_0$. Then there exists $g_0 \in G$ such that $g' =gg_0 \in G(\infty)$ which implies that
$$ g'=\left[\begin{array}{cc}a&b\\0&c\end {array}\right].$$
\noindent We may assume that $a,b,c \in A$. By Theorem \ref{Misquasi}, together with an argument used in the proof of Theorem \ref{isoto2tors}, it follows that 
$$a^2A=c^2A=acA.$$
Hence $a,c \in \FF_q$. Thus $g' \in G$ and so $ g \in G.Z(K)$. \\
For the second part $n$ is the smallest $n(>0)$ such that $g^n \in N.Z(K)$. Now $g^2 \in G.Z(K)$ by Corollary \ref{M^2}(i). If $n$ is odd then $g \in G.Z(K)$. Hence $n=2m$ is even. In addition when $Z \leq N$  $m$ divides $|G.Z(K):N.Z(K)|=|G:N|$.
\end{proof}

\noindent A special case of Corollary \ref{quasiquotcor}, combined with Corollary \ref{exquasi-inner}, is the following. 
\begin{Cor}\label{Serreeven} Suppose that $|\Cl(A)|=|\Cusp(G)|$ is even. Then there exists a 
stabilizer invariant automorphism of $G \backslash \TTT$ of order $2$. 
\end{Cor}
\noindent Serre \cite[Exercise 2(e), p.117]{Serre} states this result for the case $g(K)=0$ (i.e. $K=\FF_q(t)$) and $\delta$ even. The restriction here is necessary. For the case $g(K)=0,\;\delta=1$, in which case $A=\FF_q[t]$ and $|\Cl(A)|=1$, it is known by Nagao's Theorem \cite[Corollary, p.87]{Serre} that $G\backslash \TTT$ is a cuspidal ray whose terminal vertex is isolated. Here then the only (graph) automorphism is trivial. \\ 
\noindent Corollary \ref{Serreeven} shows that $Quinn(G)$ acts non-trivially on $G \backslash \TTT$. This extends to an action on its cuspidal rays which we now describe. We use the notation of Theorem \ref{Serre}.
\begin{Def} Let $\mathcal{R}_1,\mathcal{R}_2$ be rays in a graph $\mathcal{G}$. We write
$$\mathcal{R}_1 \sim \mathcal{R}_2$$
if and only if $|\mathcal{R}_i \backslash \mathcal{R}_1 \cap \mathcal{R}_2| < \infty\;(i=1,2)$.
This a well-known equivalence relation. The equivalence class containing the ray $\mathcal{R}$ is called the \bf end (\rm of $\mathcal{G}$)  determined by $\mathcal{R}$.
In the notation of Theorem \ref{Serre} we denote by $\mathcal{E}(e)$
the end (in $G \backslash \TTT$) determined by $\mathcal{R}(e)$, where
$e=c_i,d_j$.
\end{Def}
\noindent Now let $\kappa=\iota_g$, where $g \in N_{\widehat{G}}(G
\backslash G.Z(K)$, be a non-trivial quasi-inner automorphism and let $\widehat{\kappa}$ be the corresponding (non-trivial) element of $Quinn(G)$. Now fix $e \in \mathcal{S}$. Let $e^*=\kappa(e)$. Then by Corollary \ref{actiononcusps} $e \neq e^*$ and we may assume that $e^* \in \mathcal{S}$.
\noindent As in Theorem \ref{Serre}
$$\vert(\mathcal{R}(e))=\{\widetilde{v}_1,\widetilde{v}_2, \cdots \}\;\mathrm{and}\;\vert(\mathcal{R}(e^*))=\{\tilde{v}_1^*,\tilde{v}_2^*, \cdots \}.$$
\noindent Recall that $$\bigcup_{i \geq 1}G_{v_i}=G(e),$$
\noindent and that $G_{v_i} \leq G_{v_{i+1}}\;(i \geq 1)$. In addition it is known \cite[Theorem 2.1 (a)]{MSstabilizer} that there exists a normal subgroup $N_i$ of $G_{v_i}$ such that
$$G_{v_i}/N_i \cong \FF_q^* \times\FF_q^*,$$
\noindent where $N_i \cong V_i^+$, the additive group of a $\FF_q$-vector space of dimension $n_i$. It is also known that $n_i < n_{i+1}$. Corresponding results hold for $\mathcal{R}(e^*)$.
\noindent Now let
 $$m_X=\max \{|G_v|: v \in \vert(X)\}.$$ 
 \noindent (Recall that $X$ is \it finite.\rm ). Now choose any $m > m_X$.
By the definition of graph automorphism $\overline{\kappa_G}$ (determined by the non-trivial element $\widehat{\kappa}$ of $Quinn(G)$, together with Theorem \ref{Serre}, there exists $n > m_X$ such that
$$\overline{\kappa}_G:\tilde{v}_{m+i}\mapsto\tilde{v}_{n+i}^*\;,$$
\noindent for all $i\geq 0$. This gives rise to a map
$$\widehat{\kappa}: \mathcal{E}(e) \mapsto \mathcal{E}(e^*),$$
which in turn defines a $Quinn(G)$-action on the ends defined by the cuspidal rays in $G \backslash \TTT$. (Theorem \ref{Serre}.) Since this action coincides precisely with the action of $Quinn(G)$ on $\Cusp(G)$ the following result is an immediate consequence of Theorem \ref{transandfree}(ii), Corollary \ref{actiononcusps} and Lemma \ref{order4cusp}.
\begin{Cor}\label{cuspidalaction} With the notation of Theorem \ref{Serre},
\begin{itemize}
\item[(i)] $Quinn(G)$ acts (simultaneously) freely and transitively on
$$\{\mathcal{E}(c_1), \cdots ,\mathcal{E}(c_s)\}\;\mathrm{and}\;\{\widetilde{v} \in \vert(G \backslash \TTT): G_v \cong GL_2(\FF_q)\},$$
\item[(ii)] $Quinn(G)$ acts freely on
$$\left\{ \mathcal{E}(d_j),\mathcal{E}(\tilde{d_j}): 1\leq j \leq t' \right\},$$
\item[(iii)] $Quinn(G)$ acts on
$$\left\{ \{\mathcal{E}(d_j),\mathcal{E}(\tilde{d_j})\}: 1\leq j \leq t' \right\}.$$
\item[(iv)] Under the action of $Quinn(G)$  some $\mathcal{E}(d_j)$ is mapped to $\mathcal{E}(\tilde{d_j})$ if and only if $d_j$ has order $4$ in $\Cl(A)$.
\end{itemize}
\end{Cor} 
\noindent We recall from Proposition \ref{2,3,odd} that when $\delta$ is odd $Quinn(G)$ also acts on $\{\widetilde{v} \in \vert(G \backslash \TTT): G_v \cong \FF_{q^2}^*\}$.

\noindent Our final result in this section concerns the action of $N_{\widehat{G}}(G)$ on $\TTT$. it is known \cite[Corollary, p.75]{Serre}  that $G$ acts \it without inversion \rm (on the edges) of $\TTT$. \begin{Prop}\label{deltaodd} Suppose that $\delta$ is odd. Then every $\iota_g$ acts without inversion on $\TTT$ and hence on every quotient graph $H \backslash \TTT$.
\end{Prop}
\begin{proof} As in Theorem \ref{Misquasi} we can represent $\iota_g$ with a matrix $M$ in $\widehat{G}$
and we can assume that all its entries lie in $A$. Let $\Delta= \det(M)$. Then the $A$-ideal generated by $\Delta$ is the \it square \rm of an ideal in $A$, again by Theorem \ref{Misquasi}. It follows that, for all places $v \neq v_{\infty}$, $v(\Delta)$ is even. \\
By the product formula then $\delta v_{\infty}(\Delta)$ and hence $v_{\infty}(\Delta)$ is even. The result follows from \cite[Corollary, p.75]{Serre}.
\end{proof}

\begin{Exm} To conclude this section we consider the case where $g(K)=0$ and $\delta(K)=2$. We recall that there exists a quadratic polynomial $\pi \in \FF_q[t]$, irreducible over $\FF_q$, such that

$$ A=\left\{ \frac{f}{\pi^m}: f \in \FF_q[t],\;m \geq 0, \deg f \leq 2m \right\}.$$
In this case it is known that $\Cl(A)_2=\Cl(A) \cong Quinn(G) \cong \ZZ/2\ZZ$.
It is well-known that $G \backslash \TTT$ is a doubly infinite line, without backtracking. See \cite[ 2.4.2 (a), p.113]{Serre} and, for a more detailed description, \cite[Section 3]{MSJLMS}. It is known that $G\backslash \TTT$ lifts to a doubly infinite line $\mathcal{D}$ in $\TTT$ which we now describe in detail. For some $g_0 \in N_{\widehat{G}}(G) \backslash G.Z(K)$, $\vert(\mathcal{D})= \{v_0,v_0^*,\;v_1,v_1^*,\cdots\}$, where 
\begin{itemize}
\item[(i)] $v_i^*=g_0(v_i)\;(i \geq 0)$,
\item[(ii)] $G_{v_i^*}=(G_{v_i})^{g_0}\; (i\geq 0))$,
\item[(iii)] $G_{v_0}=GL_2(\FF_q)$,
\item[(iv)] for each $i \geq 1$$$ G_{v_i}=\left\{\left[\begin{array}{lll} \alpha&c\pi^{-i}\\[10pt]
0 & \beta\end{array}\right]:\alpha,\beta \in \FF_q^*,\deg c \leq 2i\right\}.$$ 
\end{itemize}
Then $\mathcal{D}$ maps onto (and is isomorphic to) $G\backslash \TTT$ which has the following structure. 
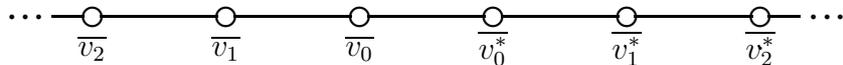
\begin{center}
	\setlength{\unitlength}{1pt} \thicklines
	\begin{picture}(340,50)(-150,-30)
	\put(-100,0){\circle{7}} \put(-50,0){\circle{7}}
	\put(0,0){\circle{7}} \put(50,0){\circle{7}}
	\put(100,0){\circle{7}} \put(-97,0){\line(1,0){44}}
	\put(-47,0){\line(1,0){44}} \put(3,0){\line(1,0){44}}
	\put(53,0){\line(1,0){44}} \put(-115,0){\line(1,0){12}}
	\put(103,0){\line(1,0){12}}
	\put(-120,0){\circle{.5}} \put(-125,0){\circle{.5}}
	\put(-130,0){\circle{.5}} \put(-100,-12){\makebox(0,0){$\overline{v_{2}}$}}
	\put(-50,-12){\makebox(0,0){$\overline{v_{1}}$}}
	\put(0,-12){\makebox(0,0){$\overline{v_0}$}}
	\put(50,-12){\makebox(0,0){$\overline{v_0^*}$}}
	\put(100,-12){\makebox(0,0){$\overline{v_1^*}$}}
	\put(103,0){\line(1,0){44}}
	\put(150,-12){\makebox(0,0){$\overline{v_2^*}$}}
	\put(150,0){\circle{7}} \put(153,0){\line(1,0){12}}
	\put(150,-12){\makebox(0,0){.}}
	\put(170,0){\circle{.5}}\put(175,0){\circle{.5}} \put(180,0){\circle{.5}}
	\end{picture}
\end{center}	
\noindent The action of the (essentially only) non-trivial quasi-inner automorphism of $G \backslash\TTT$ (represented by $g_0$) is given by 
\begin{center}$\overline{v_i} \leftrightarrow \overline{v_i^*}\;(i \geq 0).$
\end{center}

\noindent We note two features of $\mathcal{D}$ which are of interest relevant to this section. \\ 
(i) From the structure of $\mathcal{D}$ it is clear that the non-trivial quasi-inner automorphism determined by $g_0$ \it inverts \rm the edge joining $v_0$ and $v_0^*$, which shows that the restriction on $\delta$ in Proposition \ref{deltaodd} is necessary. \\
(ii) For this case there is only one stabilizer invariant involution. However the \it graph \rm $G \backslash \TTT$ has many automorphisms. Infinitely many examples include translations (which have infinite order) and reflections in any vertex (which are involutions). \\
\end{Exm}

\section{Two instructive examples}\label{sect:examples}
\noindent  We conclude with two examples which demonstrate how our results apply to the structure of the quotient graph $G \backslash \TTT$. Both are  \it elliptic \rm function fields $K/ \FF_q $. We record some of their basic properties. 
\begin{Def} A function field $K/ \FF_q $ is \bf elliptic \rm \cite[p.217]{Stich} if
 $g(K)=1$ and $K$ has a place $\infty$ of degree $1$.
\end{Def}
\begin{Thm} Suppose that $K/\FF_q$ is elliptic. Then
\begin{itemize}
\item[(i)] $$K=\FF_q(x,y),$$ where $x,y$ satisfy a (smooth) Weierstass equation $F(x,y)=0$ with 
$$F(x,y)=y^2+a_1xy+a_3y -x^3-a_2x^2-a_4x-a_6 \in \FF_q[x,y].$$ 
\item[(ii)] $\Cl^0(K)(\cong \Cl(A))$ is isomorphic to
$E(\FF_q)$, the group of \it $\FF_q$-rational points, $\{(\alpha,\beta) \in \FF_q \times \FF_q: F(\alpha,\beta)=0 \} \cup \{(\infty,\infty)\}$. Here the group operation is point addition $\oplus$ according to the chord-tangent law. 
\end{itemize}
\end{Thm}
\begin{proof} For (i) see \cite[Proposition 6.1.2]{Stich}. For (ii) see \cite[Propositions 6.1.6, 6.1.7]{Stich}.
\end{proof}
\noindent Here a rational point $(a,b) \in E(\FF_q)$ corresponds to the ideal class of $A(x-a)+A(y-b)$.
\noindent We also require some ``elliptic" properties of $\widetilde{K}=K.\FF_{q^2}$ (which is a \it constant field extension \rm of $K$).

\begin{Cor} Suppose that $K/\FF_q$ is elliptic. Then $\widetilde{K}/\FF_{q^2}$ is also elliptic and defined by the same Weierstrass equation.
\end{Cor}
\begin{proof} From the above $\widetilde{K}=\FF_{q^2}(x,y)$, where $F(x,y)=0$. The rest follows from \cite[Proposition 6.1.3]{Stich}.
\end{proof} 
\noindent With our choice of infinite place we have 
$$A=\FF_q[x,y]\;\;\mathrm{and}\;\; \widetilde{A}=\FF_{q^2}[x,y],$$
where $x$ and $y$ satisfy the Weierstrass equation $F(x,y)=0$. In an analagous way $$\mathrm{Cl}(\widetilde{A}) \cong \mathrm{Cl}^0(\widetilde{K}) \cong E(\FF_{q^2}).$$

\noindent We recall that the image of any $\alpha \in \FF_{q^2}$ under the Galois automorphism of $\FF_{q^2}/\FF_q$ is denoted by $\overline{\alpha}$. For each rational point $P=(\alpha,\beta) \in E(\FF_{q^2})$ we put $\overline{P}=(\overline{\alpha},\overline{\beta})$.
\begin{Cor} Suppose that $K/ \FF_q$ is elliptic. Under the identifications of $\Cl^0(\widetilde{K})$ (resp. $\Cl^0(K)$) with $E(\FF_{q^2})$ (resp. $E(\FF_q)$) the norm map $N:\Cl^0(\widetilde{K}) \rightarrow \Cl^0(K)$
translates to a map $N_E: E(\FF_{q^2}) \rightarrow E(\FF_q)$ defined by
$$N_E(P)=P \oplus \overline{P},$$
so that $$P \in \ker N_E \Leftrightarrow \overline{P}=-P.$$
\end{Cor}

\noindent Takahashi \cite{Takahashi} has described in detail the quotient graph for an elliptic function field over \it any \rm field of constants. In all cases $G \backslash \TTT$ is a \it tree. \rm Since $\delta=1$, for the case of a finite field of constants, the isolated vertices of $G \backslash \TTT$ are precisely those whose stabilizer is isomorphic to $GL_2(\FF_q)$ or $\FF_{q^2}^*$ by \cite[Theorem 5.1]{MSstabilizer}. For each cusp $c \in \Cl(A)_2$ the cuspidal ray $\mathcal{R}(c)$ in $G \backslash \TTT$ has attached to its terminal vertex (appearing as a ``spike") an isolated vertex with stabilizer isomorphic to $GL_2(\FF_q)$. The remaining cuspidal rays
consist of $\frac{1}{2}|\Cl(A) \backslash \Cl(A)_2|$ inverse pairs $\{\mathcal{R}(c),\mathcal{R}(c^{-1})\}$ which share a terminal vertex (appearing  in $G \backslash \TTT$ as the ``prongs" of a ``fork").\\ 
\noindent In both our examples $q=7$ in which case the Weierstrass equation can be assumed to take the short form
$$y^2=f(x)=x^3+ax+b,$$
\noindent where $a,b \in \FF_q$ and $f(x)$ has no repeated roots.

\begin{Exm}\label{x3-3x}  $K=\FF_7(x,y),\; A=\FF_7[x,y]$ with $y^2=x^3-3x$  \\ 
\noindent It can be easily shown that 
\begin{center} $ E(\FF_7)=\{(\infty,\infty),(0,0),(2,\pm 3),(3, \pm 2),(6, \pm 3)\}.$
\end{center}
\noindent Since $E$ is in short Weierstrass form the $8$ points are listed as (additive) inverse pairs. In particular $(0,0)$ is the only such $2$-torsion point. It follows that $Quinn(G) \cong \Cl(A)_2 \cong \ZZ/2\ZZ$ and hence that $\Cl(A) \cong \ZZ/8\ZZ$. Let $\kappa$ be a non-trivial quasi-inner automorphism of $G$ representing the non-trivial element of $Quinn(G)$. In $E(\FF_7)$ $\kappa$ is represented by $(0,0)$ and, by Theorem \ref{multiplicationcusp}, its action on $\Cusp(G)$ is determined by its action (via point addition $\oplus$) in $E(\FF_7)$. In a diagram of $G \backslash \TTT$ as described in \cite{Takahashi} we wish to ensure that its involution provided by $\kappa$, Corollary \ref{Serreeven}, is given by the reflection in the vertical axis. We begin by labelling appropriately its $8$ cuspidal rays (corresponding to $E(\FF_7)$). By Corollary \ref{actiononcusps} $\kappa$ acts freely on these. By Corollary \ref{Serrecor} (i) it
is clear that $\kappa$ interchanges the cusps $(\infty,\infty)$ and $(0,0)$. Attached to each of these is a "spike" consisting of an isolated vertex whose stabilizer is isomorphic to $GL_2(\FF_7)$. Since $\kappa$ is a graph automorphism it interchanges these vertices, namely, $g_1$ and $g_2$. By means of the duplication formula 
\cite[p.53]{Silver} it is easily checked that the rational $4$-torsion points are $(2, \pm 3)$. Then $\kappa$ interchanges $(2,3)$ and $(2,-3)$ by Lemma \ref{order4cusp}. For the remaining cusps $\kappa$ interchanges $(3, \pm 2)$ and $(6, \pm 3)$. To make this more precise we use the addition formulae \cite[p.53]{Silver} which show that $(0,0) \oplus (3,2)=(6,3)$. Hence $\kappa$ interchanges $(3,2)$ and $(6,3)$ by Theorem \ref{multiplicationcusp}. \\ 
\noindent There remain the isolated vertices $1$, $4$ and $5$ each of whose stabilizer is isomorphic to $\FF_{49}^*$. We deal with these via their connection with elliptic points. We recall from Theorem \ref{norm}
and the above that there exists a one-one correspondence
$$\Ell(G) \leftrightarrow \ker N_E=\{ (\alpha,\beta) \in E(\FF_{49}): (\overline{\alpha},\overline{\beta})=(\alpha,-\beta)\},$$ 
\noindent since the Weierstrass equation is in short form.
 \noindent Now let $i$ denote one of the $2$ square roots of $-1$ in $\FF_{q^2}^*$. Then
 $$N_E=\{(\rho,\epsilon i) \in E(\FF_{49}): \rho,\epsilon \in \FF_q\}.$$
 \noindent We conclude then that
 $$\Ell(G) \leftrightarrow \{(\infty,\infty), (0,0),(1,\pm3i), (4, \pm 2i), (5, \pm 3i)\}.$$

 \noindent Here $\Ell(G)$ is identified with a subgroup of $E(\FF_{49})$ listed as (additive) inverse pairs. Since there is only one $2$-torsion point $\Ell(G) \cong \ZZ/ 8\ZZ$. (In this case $|\Cl(A)|=|\Ell(G)|$. However this not a general feature. For this particular $K$ its $L$-polynomial is $L_K(t)=1+7t^2$ so that $L_K(1)=L_K(-1)$.)
 \noindent As with $\Cusp(G)$ the \it free \rm action (Corollary \ref{freeelli}) of $Quinn(G)$ on $\Ell(G)$ is represented by the action of $(0,0)$ in $N_E$ (by point addition). \\
 \noindent By identifications in Section $4$ the pairs $(1,\pm 3i),(4, \pm 2i), (5, \pm 3i)$ correspond with the vertices $1$, $4$ and $5$, respectively. By means of the duplication formula it is readily verified that the two points of order $4$ in $\Ell(G)$ are $5 \pm 3i$. By Lemma \ref{order4} it follows that $\kappa$ fixes vertex $5$ and that $\kappa$ interchanges vertices $1,4$. For a more precise version of the latter statement we note that $(0,0)\oplus(1,3i)=(4,2i)$ and so $(0,0)\oplus(1,-3i)=(4,-2i)$. \\
 \noindent It is of interest to use Theorem \ref{Misquasi} to construct a matrix $M$ which represents $\kappa$.
 We begin with the $A$-ideal, $Ax+Ay$ whose square is $Ax$. In determining a possible $M$ we recall from the proof of Theorem \ref{isoto2tors} the observation of Cremona \cite{Cremona} that every row and column of $M$ generates $\qq(M)$. Two possibilities which arise are 
 $$ M= \left[\begin{array}{cc}y&x^2\\x&y\end{array}\right]\;\mathrm{and}\; \left[\begin{array}{cc}y&-x^2\\x&-y\end{array}\right].$$

\noindent The latter is simpler since its square is a scalar matrix.

\newpage
 
{\unitlength0.5cm
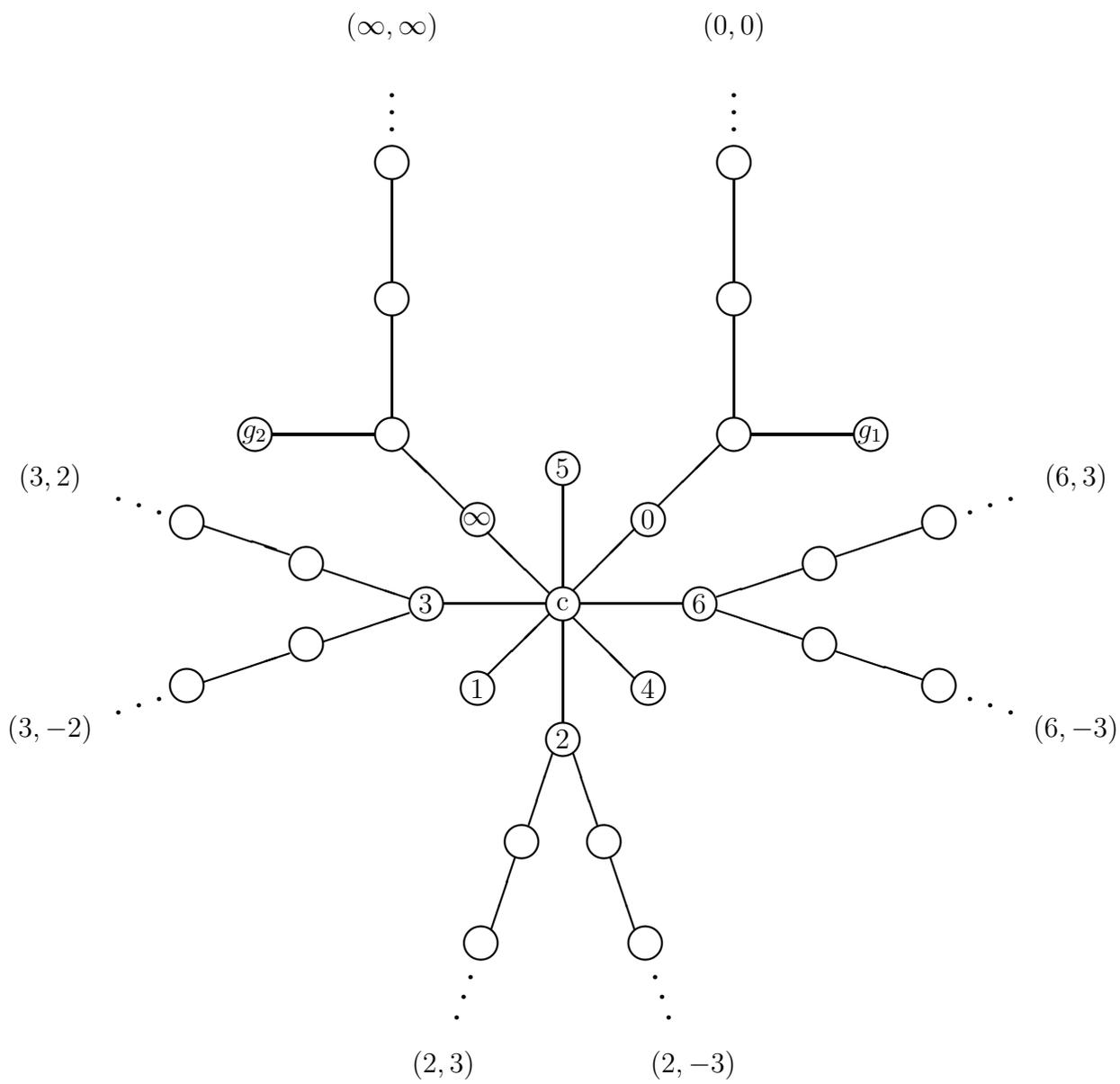
\begin{picture}(0,26)
\thicklines
\put(14,6){\circle{1.0}\makebox(-2,0){c}}
\put(14,10){\circle{1.0}\makebox(-2,0){5}}
\put(14,2){\circle{1.0}\makebox(-2,0){2}}
\put(11.5,8.5){\circle{1.0}\makebox(-2,0){$\infty$}}
\put(16.5,8.5){\circle{1.0}\makebox(-2,0){0}}
\put(10,6){\circle{1.0}\makebox(-2,0){3}}
\put(18,6){\circle{1.0}\makebox(-2,0){6}}
\put(9,11){\circle{1.0}\makebox(-2,0){}}
\put(19,11){\circle{1.0}\makebox(-2,0){}}
\put(11.5,3.5){\circle{1.0}\makebox(-2,0){1}}
\put(16.5,3.5){\circle{1.0}\makebox(-2,0){4}}
\put(23,11){\circle{1.0}\makebox(-2,0){$g_1$}}
\put(5,11){\circle{1.0}\makebox(-2,0){$g_2$}}
\put(9,15){\circle{1.0}\makebox(-2,0){}}
\put(19,15){\circle{1.0}\makebox(-2,0){}}
\put(9,19){\circle{1.0}\makebox(-2,0){}}
\put(19,19){\circle{1.0}\makebox(-2,0){}}

\put(21.5,7.2){\circle{1.0}\makebox(-2,0){}}
\put(21.5,4.8){\circle{1.0}\makebox(-2,0){}}
\put(25,8.4){\circle{1.0}\makebox(-2,0){}}
\put(25,3.6){\circle{1.0}\makebox(-2,0){}}

\put(6.5,7.2){\circle{1.0}\makebox(-2,0){}}
\put(6.5,4.8){\circle{1.0}\makebox(-2,0){}}
\put(3,8.4){\circle{1.0}\makebox(-2,0){}}
\put(3,3.6){\circle{1.0}\makebox(-2,0){}}

\put(15.2,-1){\circle{1.0}\makebox(-2,0){}}
\put(12.8,-1){\circle{1.0}\makebox(-2,0){}}
\put(16.4,-4){\circle{1.0}\makebox(-2,0){}}
\put(11.6,-4){\circle{1.0}\makebox(-2,0){}}

\put(14.5,6){\line(1,0){3}}
\put(10.5,6){\line(1,0){3}}
\put(5.5,11){\line(1,0){3}}
\put(19.5,11){\line(1,0){3}}

\put(14,6.5){\line(0,1){3}}
\put(14,2.5){\line(0,1){3}}
\put(9,11.5){\line(0,1){3}}
\put(9,15.5){\line(0,1){3}}
\put(19,11.5){\line(0,1){3}}
\put(19,15.5){\line(0,1){3}}

\put(14.3,6.4){\line(1,1){1.8}}
\put(16.8,8.9){\line(1,1){1.8}}
\put(13.6,6.3){\line(-1,1){1.8}}
\put(11.1,8.8){\line(-1,1){1.8}}
\put(11.8,3.9){\line(1,1){1.8}}
\put(14.3,5.6){\line(1,-1){1.8}}

\put(18.5,6.2){\line(3,1){2.5}}
\put(22,7.4){\line(3,1){2.5}}
\put(18.5,5.8){\line(3,-1){2.5}}
\put(22,4.6){\line(3,-1){2.5}}

\put(3.5,3.7){\line(3,1){2.5}}
\put(7,4.9){\line(3,1){2.5}}
\put(3.5,8.3){\line(3,-1){2.5}}
\put(7,7){\line(3,-1){2.5}}

\put(14.3,1.6){\line(1,-3){0.7}}
\put(15.4,-1.5){\line(1,-3){0.7}}
\put(13,-0.5){\line(1,3){0.7}}
\put(11.9,-3.6){\line(1,3){0.7}}

\put(9,20){\circle{.1}}\put(9,20.5){\circle{.1}}\put(9,21){\circle{.1}}
\put(19,20){\circle{.1}}\put(19,20.5){\circle{.1}}\put(19,21){\circle{.1}}
\put(25.9,8.7){\circle{.1}}\put(26.5,8.9){\circle{.1}}\put(27.1,9.1){\circle{.1}}
\put(25.9,3.2){\circle{.1}}\put(26.5,3){\circle{.1}}\put(27.1,2.8){\circle{.1}}
\put(2.2,8.7){\circle{.1}}\put(1.6,8.9){\circle{.1}}\put(1,9.1){\circle{.1}}
\put(2.2,3.2){\circle{.1}}\put(1.6,3){\circle{.1}}\put(1,2.8){\circle{.1}}
\put(16.7,-5){\circle{.1}}\put(16.9,-5.6){\circle{.1}}\put(17.1,-6.2){\circle{.1}}
\put(11.3,-5){\circle{.1}}\put(11.1,-5.6){\circle{.1}}\put(10.9,-6.2){\circle{.1}}

\put(9,23){\makebox(0,0){$(\infty,\infty)$}}
\put(19,23){\makebox(0,0){$(0,0)$}}
\put(-1,9.7){\makebox(0,0){$(3,2)$}}
\put(-1,2.3){\makebox(0,0){$(3,-2)$}}
\put(29,9.7){\makebox(0,0){$(6,3)$}}
\put(29,2.3){\makebox(0,0){$(6,-3)$}}
\put(10.5,-7.6){\makebox(0,0){$(2,3)$}}
\put(17.8,-7.6){\makebox(0,0){$(2,-3)$}}

\end{picture}}

\end{Exm}
\newpage

\begin{Exm}
 $K=\FF_7(x,y),\; A=\FF_7[x,y]$ with $y^2 =x^3 -x$ \\
 \noindent It is easily verified that $$E(\FF_7)=\{(\infty,\infty),(0,0),(1,0),(6,0),(4, \pm 2),(5,\pm 1)\},$$
 \noindent listed as (additive) inverse pairs. The $2$-torsion points are $(0,0),(1,0),(6,0)$  and so 
$Quinn(G) \cong \Cl(A)_2 \cong \ZZ /2\ZZ \oplus \ZZ/2\ZZ$ and 
\noindent $\Cusp(G) \cong \Cl(A) \cong \ZZ/2\ZZ \oplus \ZZ/ 4\ZZ$. 
\noindent Let the non-trivial quasi-inner automorphisms $\kappa_0,\kappa_1,\kappa_6$ represent $(0,0),(1,0),(6,0)$, respectively, where $\kappa_0=\kappa_1\kappa_6$. In the diagram representing $G \backslash \TTT$
 we label the $8$ cusps with the above rational points in such a way that (i) the action of $\kappa_6$ is the reflection about the vertical axis (ii) the action of $\kappa_1$ is reflection about the horizontal axis and (iii) (consequently) the action of $\kappa_0$ is a rotation of $180$ degrees about the ``central" vertex $c$. \\
 \noindent There are $4$ vertices whose stabilizers are isomorphic to $GL_2(\FF_7)$ which appear as ``spikes" attached to the $4$ cusps given by the $2$-torsion points in $E(\FF_7)$ and so  $\kappa_6$,   $\kappa_1$ and $\kappa_0$ interchange the vertex pairs $\{g_1,g_2\}$, $\{g_1, g_4\}$ and $\{g_1,g_3\}$, respectively.\\
 \noindent In $\Cl(A)$ there are $4$ points of order $4$, namely $(4, \pm 2)$ and $(5, \pm 1)$, and it is easily verified that the square of each is $(1,0)$. By Lemma \ref{order4cusp} it follows that $\kappa_1$ interchanges the cusps $(4,2), (4,-2)$ as well as $(5,1),(5,-1)$. On the other hand $\kappa_6$ interchanges the pairs $(4, \pm 2)$ and $(5, \pm 1)$. In more detail $\kappa_6$ maps $(5,1)$ to $(4,-2)$, since $(6,0)\oplus (5,1)=(4,-2)$. \\
 \noindent There remain two vertices $2$ and $3$ whose stabilizers are cyclic order $q^2-1$. As in the previous example we consider the elliptic function field $\widetilde{K}=K.\FF_{49}=\FF_{49}(x,y):y^2=x^3-x$. As before let $i$ denote one of the square roots of $-1$ in $\FF_{49}$. It can be verified that 
 $$ \Ell(G) \leftrightarrow N_E= \{(\infty,\infty),((0,0),((1,0),(6,0),(2, \pm i),(3,\pm 2i)\},$$
 \noindent listed as additive inverse pairs in $E(\FF_{49})$.
 \noindent As before $|\Cl(A)|=|\Ell(G)|=8$. (Again this is purely coincidental because $L_K(t)=1+7t^2$.) By correspondences discussed in Section 4 the $2$ vertices of interest here correspond to the pairs $(2, \pm i)$ and $(3, \pm 2i)$. It is easily verified that the squares of all $4$ of these points are $(6,0)$. It follows from Lemma \ref{order4} that $\kappa_6$ fixes $2$ and $3$. On the other hand $(1,0)\oplus (2,i)=(3,-2i)$ and so $\kappa_1$ interchanges $2$ and $3$.
 \noindent Finally using Theorem \ref{Misquasi} the following matrices $M_0,\;M_1,\;M_6=M_0M_1$ represent 
 $\kappa_0, \kappa_1, \kappa_6$, respectively,
  $$ M_0= \left[\begin{array}{cc}y&-x^2\\x&-y\end{array}\right]\;\mathrm{and}\; M_1=\left[\begin{array}{cc}y&-(x-1)(x+2)\\x-1&-y\end{array}\right].$$

 \newpage
 
{\unitlength0.5cm
\begin{picture}(0,26)
\thicklines
\put(14,6){\circle{1.0}\makebox(-2,0){c}}
\put(14,10){\circle{1.0}\makebox(-2,0){2}}
\put(14,2){\circle{1.0}\makebox(-2,0){3}}
\put(11.5,8.5){\circle{1.0}\makebox(-2,0){$\infty$}}
\put(16.5,8.5){\circle{1.0}\makebox(-2,0){6}}
\put(10,6){\circle{1.0}\makebox(-2,0){4}}
\put(18,6){\circle{1.0}\makebox(-2,0){5}}
\put(9,11){\circle{1.0}\makebox(-2,0){}}
\put(19,11){\circle{1.0}\makebox(-2,0){}}
\put(11.5,3.5){\circle{1.0}\makebox(-2,0){1}}
\put(9,1){\circle{1.0}\makebox(-2,0){}}
\put(5,1){\circle{1.0}\makebox(-2,0){$g_4$}}
\put(9,-3){\circle{1.0}\makebox(-2,0){}}
\put(9,-7){\circle{1.0}\makebox(-2,0){}}

\put(16.5,3.5){\circle{1.0}\makebox(-2,0){0}}
\put(19,1){\circle{1.0}\makebox(-2,0){}}
\put(23,1){\circle{1.0}\makebox(-2,0){$g_3$}}
\put(19,-3){\circle{1.0}\makebox(-2,0){}}
\put(19,-7){\circle{1.0}\makebox(-2,0){}}

\put(23,11){\circle{1.0}\makebox(-2,0){$g_2$}}
\put(5,11){\circle{1.0}\makebox(-2,0){$g_1$}}
\put(9,15){\circle{1.0}\makebox(-2,0){}}
\put(19,15){\circle{1.0}\makebox(-2,0){}}
\put(9,19){\circle{1.0}\makebox(-2,0){}}
\put(19,19){\circle{1.0}\makebox(-2,0){}}

\put(21.5,7.2){\circle{1.0}\makebox(-2,0){}}
\put(21.5,4.8){\circle{1.0}\makebox(-2,0){}}
\put(25,8.4){\circle{1.0}\makebox(-2,0){}}
\put(25,3.6){\circle{1.0}\makebox(-2,0){}}

\put(6.5,7.2){\circle{1.0}\makebox(-2,0){}}
\put(6.5,4.8){\circle{1.0}\makebox(-2,0){}}
\put(3,8.4){\circle{1.0}\makebox(-2,0){}}
\put(3,3.6){\circle{1.0}\makebox(-2,0){}}

\put(14.5,6){\line(1,0){3}}
\put(10.5,6){\line(1,0){3}}
\put(5.5,11){\line(1,0){3}}
\put(19.5,11){\line(1,0){3}}
\put(5.5,1){\line(1,0){3}}
\put(19.5,1){\line(1,0){3}}

\put(14,6.5){\line(0,1){3}}
\put(14,2.5){\line(0,1){3}}
\put(9,11.5){\line(0,1){3}}
\put(9,15.5){\line(0,1){3}}
\put(19,11.5){\line(0,1){3}}
\put(19,15.5){\line(0,1){3}}
\put(9,-6.5){\line(0,1){3}}
\put(9,-2.5){\line(0,1){3}}
\put(19,-6.5){\line(0,1){3}}
\put(19,-2.5){\line(0,1){3}}

\put(14.3,6.4){\line(1,1){1.8}}
\put(16.8,8.9){\line(1,1){1.8}}
\put(13.6,6.3){\line(-1,1){1.8}}
\put(11.1,8.8){\line(-1,1){1.8}}
\put(11.8,3.9){\line(1,1){1.8}}
\put(9.3,1.4){\line(1,1){1.8}}
\put(14.3,5.6){\line(1,-1){1.8}}
\put(16.8,3.1){\line(1,-1){1.8}}

\put(18.5,6.2){\line(3,1){2.5}}
\put(22,7.4){\line(3,1){2.5}}
\put(18.5,5.8){\line(3,-1){2.5}}
\put(22,4.6){\line(3,-1){2.5}}

\put(3.5,3.7){\line(3,1){2.5}}
\put(7,4.9){\line(3,1){2.5}}
\put(3.5,8.3){\line(3,-1){2.5}}
\put(7,7){\line(3,-1){2.5}}

\put(9,20){\circle{.1}}\put(9,20.5){\circle{.1}}\put(9,21){\circle{.1}}
\put(19,20){\circle{.1}}\put(19,20.5){\circle{.1}}\put(19,21){\circle{.1}}
\put(9,-8){\circle{.1}}\put(9,-8.5){\circle{.1}}\put(9,-9){\circle{.1}}
\put(19,-8){\circle{.1}}\put(19,-8.5){\circle{.1}}\put(19,-9){\circle{.1}}
\put(25.9,8.7){\circle{.1}}\put(26.5,8.9){\circle{.1}}\put(27.1,9.1){\circle{.1}}
\put(25.9,3.2){\circle{.1}}\put(26.5,3){\circle{.1}}\put(27.1,2.8){\circle{.1}}
\put(2.2,8.7){\circle{.1}}\put(1.6,8.9){\circle{.1}}\put(1,9.1){\circle{.1}}
\put(2.2,3.2){\circle{.1}}\put(1.6,3){\circle{.1}}\put(1,2.8){\circle{.1}}

\put(9,23){\makebox(0,0){$(\infty,\infty)$}}
\put(19,23){\makebox(0,0){$(6,0)$}}
\put(9,-11){\makebox(0,0){$(1,0)$}}
\put(19,-11){\makebox(0,0){$(0,0)$}}
\put(-1,9.7){\makebox(0,0){$(4,2)$}}
\put(-1,2.3){\makebox(0,0){$(4,-2)$}}
\put(29,9.7){\makebox(0,0){$(5,-1)$}}
\put(29,2.3){\makebox(0,0){$(5,1)$}}

\end{picture}}
\end{Exm}

\end{document}